\documentclass[10pt]{scrartcl}
\usepackage{fullpage, amsmath, amsthm,amsfonts,url,amssymb,stmaryrd, mathrsfs}
\usepackage{mathrsfs}
\usepackage{amssymb,amscd}
\usepackage{color}
\usepackage{chemarrow}
\usepackage{pictexwd,dcpic}
\usepackage{extarrows}
\usepackage{hyperref}
\usepackage{enumitem}

\makeatletter
\def\ps@pprintTitle{%
 \let\@oddhead\@empty
 \let\@evenhead\@empty
 \def\@oddfoot{\centerline{\thepage}}%
 \let\@evenfoot\@oddfoot}
\makeatother

\newtheorem{theorem}{Theorem}[section]
\newtheorem{corollary}[theorem]{Corollary}
\newtheorem{lemma}[theorem]{Lemma}
\newtheorem{proposition}[theorem]{Proposition}

\newtheorem{definition}[theorem]{Definition}

\newtheorem{remark}[theorem]{Remark}
\newtheorem{notation}[theorem]{Notation}

\newcommand{\hooklongrightarrow}{\lhook\joinrel\longrightarrow}
\newcommand{\twoheadlongrightarrow}{\relbar\joinrel\twoheadrightarrow}

\newcommand{\ra}{\rightarrow}
\newcommand{\lra}{\longrightarrow}

\newcommand{\bA}{\mathbb A}
\newcommand{\bC}{\mathbb C}

\newcommand{\bG}{\mathbb G}

\newcommand{\Q}{\mathbb Q}
\newcommand{\bR}{\mathbb R}

\newcommand{\Z}{\mathbb Z}

\newcommand{\bP}{\mathbb P}

\newcommand{\co}{\mathcal O}

\newcommand{\cH}{\mathcal H}
\newcommand{\cC}{\mathcal C}
\newcommand{\cS}{\mathcal S}

\newcommand{\cI}{\mathcal I}
\newcommand{\cW}{\mathcal W}

\newcommand{\cM}{\mathcal M}
\newcommand{\cF}{\mathcal F}
\newcommand{\cV}{\mathcal V}

\newcommand{\cU}{\mathcal U}

\newcommand{\fm}{\mathfrak{m}}
\newcommand{\ub}{\mathfrak b}

\newcommand{\ug}{\mathfrak g}

\newcommand{\ft}{\mathfrak t}

\newcommand{\sL}{\mathscr L}

\DeclareMathOperator{\la}{\mathrm la}

\DeclareMathOperator{\GL}{\mathrm GL}

\DeclareMathOperator{\et}{\text{\'et}}

\DeclareMathOperator{\Frac}{\mathrm Frac}
\DeclareMathOperator{\Sym}{\mathrm Sym}
\DeclareMathOperator{\Gal}{\mathrm Gal}
\DeclareMathOperator{\Hom}{\mathrm Hom}

\DeclareMathOperator{\End}{\mathrm End}

\DeclareMathOperator{\an}{\mathrm an}

\DeclareMathOperator{\Frob}{\mathrm Frob}
\DeclareMathOperator{\Ind}{\mathrm Ind}
\DeclareMathOperator{\unr}{\mathrm unr}

\DeclareMathOperator{\lp}{\mathrm lp}

\DeclareMathOperator{\Spm}{\mathrm Spm}

\DeclareMathOperator{\lalg}{\mathrm lalg}
\DeclareMathOperator{\cl}{\mathrm cl}

\DeclareMathOperator{\dett}{\mathrm det}
\DeclareMathOperator{\alg}{\mathrm alg}

\DeclareMathOperator{\red}{\mathrm red}

\DeclareMathOperator{\wt}{\mathrm wt}

\DeclareMathOperator{\SO}{\mathrm SO}
\DeclareMathOperator{\Tr}{\mathrm Tr}
\DeclareMathOperator{\sph}{\mathrm sph}
\DeclareMathOperator{\Fra}{\mathrm Frac}
\DeclareMathOperator{\val}{\mathrm val}
\DeclareMathOperator{\sm}{\mathrm sm}
\begin{document}
\title{A note on critical $p$-adic $L$-functions}
\author{Yiwen Ding \\ yiwen.ding@bicmr.pku.edu.cn}
\date{}
\maketitle\begin{abstract}
We study the adjunction property of the Jacquet-Emerton functor in certain neighborhoods of critical points in the eigencurve. As an application, we construct two-variable $p$-adic $L$-functions around critical points via Emerton's representation theoretic approach.
\end{abstract}
\numberwithin{equation}{section}

\tableofcontents
\section{Introduction}\addtocontents{toc}{\protect\setcounter{tocdepth}{1}}
Let $N\in \Z_{>1}$, $p\nmid N$, $k\in \Z_{\geq 0}$, and $f$ be a classical newform of level $\Gamma_1(N)$ of weight $k+2$ over $E$ (which is a finite extension of $\Q_p$ sufficiently large). Let $a_p$ (resp. $b_p$) be the eigenvalue of the Hecke operator $T_p$ (resp. $S_p$) on $f$, and  $\alpha$ be a root of the Hecke polynomial $X^2-a_p X+pb_p$. To $f$ and $\alpha$, one can associate an  eigenform $f_{\alpha}$ of level $\Gamma_1(N)\cap \Gamma_0(p)$ satisfying that $f_{\alpha}$ has the same prime-to-$p$ Hecke eigenvalues as $f$, and that $U_p(f_{\alpha})=\alpha f_{\alpha}$. The eigenform $f_{\alpha}$ is called a \emph{refinement} (or a \emph{$p$-stabilization}) of $f$. We have $\val_p(a_p)\geq 0$ and $\val_p(b_p)=k$ (where $\val_p$ the additive $p$-adic valuation on $\overline{\Q_p}$ normalized with $\val_p(p)=1$). Since $\alpha^2-a_p \alpha+pb_p=0$,  we easily deduce $\val_p(\alpha)\leq k+1$. The refinement $f_{\alpha}$ is called
\begin{itemize}
  \item \emph{of non-critical slope} if $\val_p(\alpha)<k+1$,
  \item \emph{of critical slope} if $\val_p(\alpha)=k+1$,
  \item \emph{critical} if $\rho_{f,p}$ is split (which implies $\val_p(\alpha)=k+1$),
\end{itemize}where $\rho_{f,p}$ is the $2$-dimensional $\Gal_{\Q_p}$-representation associated to $f$.

To the form $f_{\alpha}$, we can associate a $p$-adic $L$-function $L(f_{\alpha}, -)$, which is a distribution on $\Z_p^{\times}$ (for example see \cite{MTT} \cite{PSc} \cite{Bel} \cite{LLZ}... see also Proposition \ref{intpl} of this note). When $f_{\alpha}$ is non-critical, $L(f_{\alpha},-)$ interpolates the critical values of the classical $L$-function attached to $f_{\alpha}$ (e.g. see (\ref{equ: clp-int})). If  $f_{\alpha}$ is  moreover of non-critical slope, we know $L(f_{\alpha}, -)$ can be determined (up to non-zero scalars) by the interpolation property. When $f_{\alpha}$ is of critical slope, the interpolation property is however not enough to determine $L(f_{\alpha}, -)$.  And in this case,  the problem of proving that the $p$-adic $L$-functions constructed by different methods coincide is not completely settled yet (to the author's knowledge).
We remark that when $f_{\alpha}$ is critical, then $L(f_{\alpha}, \phi x^j)=0$ for all smooth characters $\phi$ on $\Z_p^{\times}$ and $j\in \{0,\cdots, k\}$ where $x^j$ denotes the algebraic character $a\mapsto a^j$ on $\Z_p^{\times}$ (e.g. see Proposition \ref{cri0}).   In all cases,  $L(f_{\alpha}, -)$ fits  into the  so-called two-variable $p$-adic L-functions $L(-,-)$ where the first variable runs around points on the eigencurve (we recall that such points correspond to overconvergent eigenforms $g$), and $L(g,-)$ is, up to non-zero scalars, equal to the $p$-adic $L$-function of $g$ when  $g$ is classical.

In \cite{Em05}, Emerton provided a representation theoretic approach of the construction of $p$-adic $L$-functions, using results in the $p$-adic Langlands program.   Via this approach, in \cite[\S~4.5]{Em1}, Emerton also constructed two-variable $p$-adic $L$-functions in certain neighborhoods of non-critical points in the eigencurve (we call a point critical if its associated eigenform is critical)). A key ingredient in this construction is an adjunction property of the Jacquet-Emerton functor around non-critical points.  In this note, we study the adjunction property around critical points in the eigencurve. We show that by adding ``poles", the adjunction can extend to critical points (cf. Theorem \ref{thm: clp-rps} and Theorem \ref{adj003}). Using this result, together with the smoothness of the eigencurve at critical points (due to Bella\"iche \cite{Bel}), we construct two-variable $p$-adic $L$-functions (see (\ref{padicL})) in some neighborhoods of  critical points via Emerton's approach. Evaluating the two-variable $p$-adic $L$-functions at the critical points then allows us to associate $p$-adic $L$-functions to the corresponding critical eigenforms.

The note is organised as follows. In \S~\ref{sec: clp-1.1}, \S \ref{sec: clp-1.2}, we recall some basic facts on the completed $H^1$ of modular curves,  and recall Emerton's construction of the eigencurves. Nothing is new in these twos sections. In \S~\ref{sec: clp-1.3}, we show an adjunction property of the Jacquet-Emerton functor in neighborhoods of a critical point on the eigencurve. In \S~\ref{sec: 21}, we use this adjunction property to construct two-variable $p$-adic $L$-functions around critical points. Finally, we study some properties of our $p$-adic $L$-functions in \S~\ref{sec: 22}.

\subsection{Notations}
In this note, $E$ will be a finite extension of $\Q_p$ with $\co_E$ its ring of integers, $\varpi_E\in \co_E$ a uniformiser and $k_E:=\co_E/\varpi_E$ its residue field.
We let $B$ denote the Borel subgroup of $\GL_2$ of upper triangular matrices, $T\subseteq B$ the subgroup of diagonal matrices,  and  $N\subseteq B$ the unipotent radical.
We use $\ft$ to denote the Lie algebra of $T(\Q_p)$ over $E$, $\ug$ the Lie algebra of $\GL_2(\Q_p)$ over $E$.
For a continuous character $\chi: T(\Q_p) \ra E^{\times}$, we denote by $d\chi: \ft \ra E$ the induced $E$-linear map with
\begin{equation*}
  d\chi (X)=\lim_{t\ra 0} \frac{\chi(\exp{tX})-1}{t} |_{t=0},
\end{equation*}
for $X\in \ft$.
Similarly, for a continuous character $\chi: \Q_p^{\times} \ra E^{\times}$, we denote by $\wt(\chi): \Q_p \ra E$ the induced $E$-linear map, called the weight of $\chi$. For $a\in E^{\times}$, we denote by $\unr(a): \Q_p^{\times} \ra E^{\times}$ the unramified character (i.e. $\unr(a)|_{\Z_p^{\times}}=1$) with $\unr(a)(p)=a$.

 For an $E$-vector space $V$ that is equipped with an $E$-linear action of $A$ (with $A$ a set of operators), and for a system of eigenvalues $\chi$ of $A$,  we denote by $V[A=\chi]$ the $\chi$-eigenspace for $A$, and $V\{A=\chi\}$ the generalized $\chi$-eigenspace for $A$.
\subsection*{Acknowledgement}
I want to thank Matthew Emerton for suggesting the problem of extending the adjunction formula to critical points on the eigencurve, that led to the note. I  thank Daniel Barrera Salazar,  John Bergdall, Xin Wan, Shanwen Wang for helpful discussions or remarks. I also thank the anonymous referee for the reading and helpful suggestions. This work was supported by EPSRC grant EP/L025485/1 and by Grant No. 7101500268 from Peking University.

\addtocontents{toc}{\protect\setcounter{tocdepth}{2}}
\section{Eigencurves}
\subsection{Completed cohomology of modular curves}\label{sec: clp-1.1}Let $\bA^{\infty}$ denote the finite adeles of $\Q$. For a compact open subgroup $K$ of $\GL_2(\bA^{\infty})$, we denote by $Y_K$ the affine modular curve over $\Q$ such that the $\bC$-points of $Y_K$ are given by
\begin{equation*}
  Y_K(\bC)=\GL_2(\Q)\backslash \GL_2(\bA)/\big(\bR_+^{\times}\SO_2(\bR)K\big).
\end{equation*}
Let $K^p=\prod_{\ell\neq p} K_{\ell}$ be a compact open subgroup of $\GL_2(\bA^{\infty,p})$, and
\begin{equation*}\Sigma(K^p):=\{p\}\cup \{\ell\neq p, \ \text{$K_{\ell}$ is not maximal}\}.
\end{equation*}Let  $\cH(K^p)$ denote the Hecke $\co_E$-algebra of $K^p$ double cosets in $\GL_2(\bA^{\infty,p})$ (which is non-commutative in general). Following Emerton (cf. \cite[(2.1.1)]{Em1}), we put
\begin{equation*}
\widetilde{H}^1_{\et,c}(K^p,\co_E):=\varprojlim_{n} \varinjlim_{K_p} H^1_{\et,c}\big(Y_{K_pK^p,\overline{\Q_p}}, \co_E/\varpi_E^n\big)
\end{equation*}
which is a complete $\co_E$-module equipped with a continuous action of $\GL_2(\Q_p)\times \cH(K^p)\times \Gal_{\Q}\times \pi_0$, where $\pi_0$ is the $2$-element group generated by the archimedean Hecke operator $$T_{\infty}:=(\bR^{\times}_+ \SO_2(\bR))\begin{pmatrix} 1 & 0 \\ 0 & -1\end{pmatrix}\SO_2(\bR).$$
Let $\cH^p$ be a commutative $\co_E$-subalgebra of $\cH(K^p)$ containing $\cH(K^p)^{\sph}:=\otimes'_{\ell\notin \Sigma(K^p)} \co_E[T_{\ell},S_{\ell}^{\pm}]$ with $T_{\ell}:=K^p\begin{pmatrix} \ell & 0 \\ 0 & 1 \end{pmatrix}K^p$, $S_{\ell}:=K^p\begin{pmatrix}\ell & 0 \\ 0 & \ell \end{pmatrix}K^p$. One has the Eichler-Shimura relations on $\widetilde{H}^1_{\et,c}(K^p,\co_E)$ (e.g. see \cite[Prop. 3.2.3]{BE} and the proof):
\begin{equation*}
  \Frob_{\ell}^{-2}-T_{\ell} \Frob_{\ell}^{-1}+\ell S_{\ell}=0
\end{equation*}
for $\ell\notin \Sigma(K^p)$, where $\Frob_{\ell}$ denotes an arithmetic Frobenius at $\ell$. Put $$\widetilde{H}^1_{\et,c}(K^p,E):=\widetilde{H}^1_{\et,c}(K^p,\co_E)\otimes_{\co_E} E,$$
 which is a unitary admissible $E$-Banach space representation of $\GL_2(\Q_p)$ equipped with an action of $\cH^p \times \Gal_{\Q}\times \pi_0$, commuting with $\GL_2(\Q_p)$.

\begin{notation}(1) For an $\co_E$-module $M$ equipped with a $\pi_0$-action, denote by $M^{\pm}$ the $\pm 1$-eigenspace for $T_{\infty}$, one has thus $M=M^{+}\oplus M^-$.

(2) Let $\overline{\rho}$ be a $2$-dimensional continuous representation of $\Gal_{\Q}$ over $k_E$, unramified outside $\Sigma(K^p)$, we denote by $\fm(\overline{\rho})$ the maximal ideal of  $\cH(K^p)^{\sph}$ attached to $\overline{\rho}$ satisfying that the quotient map $\cH(K^p)^{\sph}\twoheadrightarrow \cH(K^p)^{\sph}/\fm(\overline{\rho})\cong k_E$ sends $T_{\ell}$ to $\Tr(\overline{\rho}(\Frob_{\ell}^{-1}))$ and $S_{\ell}$ to $\ell^{-1}\det(\overline{\rho}(\Frob_{\ell}^{-1}))$ for $\ell\notin \Sigma(K^p)$.

(3) For a finite $\varpi_E$-torsion $\co_E$-module $M$ that is equipped with an $\cH(K^p)^{\sph}$-action, denote by $M_{\overline{\rho}}$ the localisation of $M$ at $\fm(\overline{\rho})$. We put
\begin{equation*}
\widetilde{H}^1_{\et,c}(K^p,\co_E)_{\overline{\rho}}:=\varprojlim_{n} \varinjlim_{K_p} H^1_{\et,c}\big(Y_{K_pK^p,\overline{\Q_p}}, \co_E/\varpi_E^n\big)_{\overline{\rho}},
\end{equation*}
and $\widetilde{H}^1_{\et,c}(K^p,E)_{\overline{\rho}}:=\widetilde{H}^1_{\et,c}(K^p,\co_E)_{\overline{\rho}} \otimes_{\co_E} E$.
\end{notation}
In the sequel, we fix a $2$-dimensional continuous representation $\overline{\rho}$ of $\Gal_{\Q}$ over $k_E$, which is absolutely irreducible, unramified outside $\Sigma(K^p)$ and satisfies $\widetilde{H}^1_{\et,c}(K^p,E)_{\overline{\rho}}\neq 0$.
We summarize some properties of $\widetilde{H}^1_{\et,c}(K^p,E)_{\overline{\rho}}$:
\begin{theorem}\label{thm: clp-cco}

(1) $\widetilde{H}^1_{\et,c}(K^p,E)_{\overline{\rho}}$ is a direct summand of $\widetilde{H}^1_{\et,c}(K^p,E)$ stable under $\GL_2(\Q_p)\times \Gal_{\Q}\times \cH^p\times \pi_0$. For any compact open pro-$p$-subgroup $K_p$ of $\GL_2(\Z_p)$, there exists $r\in \Z_{>0}$ such that we have an isomorphism of $K_p$-representations:
\begin{equation}
  \widetilde{H}^1_{\et,c}(K^p,E)_{\overline{\rho}}|_{K_p}\cong \cC(K_p,E)^{\oplus r},
\end{equation}
where $\cC(K_p,E)$ denotes the space of continuous functions on $K_p$ with values in $E$, equipped with the right regular $K_p$-action.

(2) We have:
\begin{equation}\label{localg}
  \widetilde{H}^1_{\et,c}(K^p,E)_{\overline{\rho}}^{\lalg}\xlongrightarrow{\sim} \oplus_{k\in \Z_{\geq 2}, w\in \Z} H^1_{\et,c}(K^p,\cF_{\alg(k,w)})_{\overline{\rho}} \otimes_E \alg(k,w)^{\vee}.
\end{equation}
where $\widetilde{H}^1_{\et,c}(K^p,E)_{\overline{\rho}}^{\lalg}$ denotes the locally algebraic vectors for $\GL_2(\Q_p)$ inside $\widetilde{H}^1_{\et,c}(K^p,E)_{\overline{\rho}}$, $\alg(k,w)$ denotes the algebraic representation $\Sym^{k-2} E^2 \otimes_E \dett^w$  of $\GL_2(\Q_p)$, $\cF_{\alg(k,w)}$ denotes the local system associated to $\alg(k,w)$ on $Y_{K_pK^p}$ for compact open subgroups $K_pK^p$ of $\GL_2(\bA^{\infty})$, and where $$H^1_{\et,c}(K^p,\cF_{\alg(k,w)})_{\overline{\rho}}:=\varinjlim_{K_p} H^1_{\et,c}(Y_{K^pK_p,\overline{\Q}},\cF_{\alg(k,w)})_{\overline{\rho}}$$ denotes the (classical) \'etale cohomology with compact support of modular curves (localized at $\fm(\overline{\rho})$) with tame level $K^p$ and coefficients in $\cF_{\alg(k,w)}$ (with $K_p$ running through compact open subgroups of $\GL_2(\Z_p)$).
\end{theorem}
\begin{proof}
  For (1), see the discussion in \cite[\S~5.3]{Em4} and \cite[Cor. 5.3.19]{Em4}. For (2), see \cite[Thm. 4.1]{Br11b} (see also \cite[Cor. 2.2.18, (4.3.4)]{Em1}).
\end{proof}
\subsection{Eigencurves (eigensurfaces)}\label{sec: clp-1.2}
We first briefly recall Emerton's construction of the  eigencurves (eigensurfaces), and we refer to \cite[\S~2.3]{Em1} for details. We let  $\widetilde{H}^1_{\et,c}(K^p,E)_{\overline{\rho}}^{\an}$ be the locally analytic subrepresentation of $\widetilde{H}^1_{\et,c}(K^p,E)_{\overline{\rho}}$. By \cite[Thm. 0.5]{Em11}, applying the Jacquet-Emerton functor to $\widetilde{H}^1_{\et,c}(K^p,E)_{\overline{\rho}}^{\an}$, one gets an essentially admissible locally analytic representation $J_B\big(\widetilde{H}^1_{\et,c}(K^p,E)_{\overline{\rho}}^{\an}\big)$ of $T(\Q_p)$. Here $J_B\big(\widetilde{H}^1_{\et,c}(K^p,E)_{\overline{\rho}}^{\an}\big)$ being essentially admissible means that there exists a coherent sheaf $\cM$ over $\widehat{T}$ such that (cf. \cite[\S 6.4]{Em04}, \cite[Prop. 2.3.2]{Em1})
\begin{equation}\label{jacEss}\cM(\widehat{T})\cong J_B\big(\widetilde{H}^1_{\et,c}(K^p,E)_{\overline{\rho}}^{\an}\big)^{\vee}_b,
\end{equation}
where ``$-^{\vee}_b$"  denotes the continuous dual equipped with the strong topology, and where $\widehat{T}$ denotes the rigid space parametrizing locally analytic characters of $T(\Q_p)$ (e.g. see \cite[Prop. 6.4.5]{Em04}).
Moreover, $J_B\big(\widetilde{H}^1_{\et,c}(K^p,E)_{\overline{\rho}}^{\an}\big)$ inherits from $\widetilde{H}^1_{\et,c}(K^p,E)_{\overline{\rho}}^{\an}$ a continuous action of $\cH^p\times \pi_0$ commuting with $T(\Q_p)$. Hence $\cM$  is equipped with a natural $\co(\widehat{T})$-linear action of
$\cH^p\times \pi_0$ such that the isomorphism in (\ref{jacEss}) is $\cH^p\times \pi_0$-equivariant.

From $\{\cM,\widehat{T},\cH^p\}$, one can  construct as in \cite[\S~2.3]{Em1} a rigid analytic space $\cS$ over $E$ equipped with a natural finite morphism $\kappa_1: \cS\ra \widehat{T}$ such that
for any admissible affinoid open $U=\Spm A\subset \widehat{T}$, we have $\kappa_1^{-1}(\cU)\cong \Spm B$ where $B$ is the finite $A$-subalgebra of $\End_A(\cM(\cU))$ generated by $\cH^p$ \big(noting that $\cM(\cU)$ is equipped with an $A$-linear $\cH^p$-action\big). The rigid space $\cS$ is referred to as the eigensurface of tame level $K^p$.  An $E$-point $z$ of $\cS$ can be parametrized as $(\chi_z,\lambda_z)$ where $\chi_z$ is a locally analytic character of $T(\Q_p)$ over $E$, and $\lambda_z: \cH^p \ra E$ is a system of eigenvalues of $\cH^p$. Moreover,  such a point $(\chi_z,\lambda_z)$ lies in $\cS$ if and only if the corresponding eigenspace
\begin{equation*}
  J_B\big(\widetilde{H}^1_{\et,c}(K^p,E)_{\overline{\rho}}^{\an}\big)[T(\Q_p)=\chi_z,\cH^p=\lambda_z]\neq 0.
\end{equation*}
The $\co(\widehat{T})$-module $\cM$ has a natural $\co(\cS)$-action, which makes $\cM$ to be a coherent $\co(\cS)$-module. For any $z=(\chi_z,\lambda_z)\in \cS$ with $k_z$ the residue field, we have a natural isomorphism of finite dimensional $k_z$-vector spaces (cf. \cite[Prop. 2.3.3 (iii)]{Em1})
\begin{equation}\label{fib0}
(z^*\cM)^{\vee} \cong J_B\big(\widetilde{H}^1_{\et,c}(K^p,E)_{\overline{\rho}}^{\an}\big)[T(\Q_p)=\chi_z,\cH^p=\lambda_z].
\end{equation}
Since the action of the center $\Q_p^{\times}$ of $\GL_2(\Q_p)$ on $\widetilde{H}^1_{\et,c}(K^p,E)_{\overline{\rho}}$ is unitary, we easily see
\begin{equation}\label{center}
\val_p(\chi_z(p))=0
\end{equation}
if $(\chi_z, \lambda_z)\in \cS$. The following definition is standard.
\begin{definition}\label{class00} (1) A point $z=(\chi_z,\lambda_z)\in \cS$ is called classical if
\begin{equation*}
  J_B\big(\widetilde{H}^1_{\et,c}(K^p,E)_{\overline{\rho}}^{\lalg}\big)[T(\Q_p)=\chi_z,\cH^p=\lambda_z]\neq 0.
\end{equation*}
A point $z$ is called very classical if $z$ is classical and the natural injection
\begin{equation*}
  J_B\big(\widetilde{H}^1_{\et,c}(K^p,E)_{\overline{\rho}}^{\lalg}\big)[T(\Q_p)=\chi_z,\cH^p=\lambda_z]
  \hooklongrightarrow J_B\big(\widetilde{H}^1_{\et,c}(K^p,E)_{\overline{\rho}}^{\an}\big)[T(\Q_p)=\chi_z,\cH^p=\lambda_z]
\end{equation*}
is bijective.

(2) Let $z=(\chi_z, \lambda_z)$ be a point in $\cS$ with $\chi_z=(\psi_{z,1}x^{k_1})\otimes (\psi_{z,2} x^{k_2})$ where $\psi_{z,i}$ are smooth characters of $\Q_p^{\times}$, $k_1$, $k_2\in \Z$ and $k_2\geq k_2$ (we call such character locally algebraic of dominant weight). We call $z$  of non-critical slope if $\val_p(p\psi_{z,1}(p))<1-k_2$.
 \end{definition}
We have by \cite[Prop. 2.3.6]{Em1}:
\begin{proposition}\label{prop: clp-cla}
 Let $z=(\chi_z,\lambda_z)\in \cS$ with $\chi_z$ locally algebraic of dominant weight. If $z$ is of non-critical slope, then $z$ is very classical.
\end{proposition}

Using Theorem \ref{thm: clp-cco} (1) and \cite[Prop. 4.2.36]{Em11}, one can actually reformulate the construction of $\cS$ using the spectral theory of compact operators (e.g. see \cite[Lem. 3.10]{BHS1}).
Let $\widehat{T}_0$  be the rigid space over $E$ parametrizing locally analytic characters of $T(\Z_p)$, and we denote by $\kappa$ the composition
\begin{equation*}
  \kappa: \cS \xlongrightarrow{\kappa_1} \widehat{T} \lra \widehat{T}_0.
\end{equation*}
Let $\varpi_1:=\begin{pmatrix}
  p & 0 \\ 0 & 1 \end{pmatrix}$,  $\varpi_2:=\begin{pmatrix}
    p & 0 \\ 0 & p
  \end{pmatrix} \in T(\Q_p)$.
By the same argument as in the proof of \cite[Prop. 3.11]{BHS1}, we have:
\begin{proposition}\label{prop: eigna}
  There exists an admissible covering $\{\cU_i\}_{i\in I}$ of $\cS$ by affinoids $\cU_i$ such that for all $i$ there exits an open affinoid $W_i$ of $\widehat{T}_0$ such that
\begin{itemize}
  \item the morphism $\kappa$ induces a finite surjective morphism from each irreducible component of $\cU_i$ onto $W_i$,
  \item $\cM(\cU_i)$ is a finite projective $\co(W_i)$-module equipped with $\co(W_i)$-linear operators $\varpi_1$, $\varpi_2$ and $T\in \cH^p$,
  \item $\co(\cU_i)$ is isomorphic to a $\co(W_i)$-subalgebra of $\End_{\co(W_i)}(\cM(\cU_i))$ generated by $\varpi_1$, $\varpi_2$ and the operators in $\cH^p$.
\end{itemize}
\end{proposition}
%
%
%
We assume that the $\cH^p$-action on $\widetilde{H}^1_{\et,c}(K^p,E)_{\overline{\rho}}^{\lalg}$ is semi-simple \big(e.g. this holds when $\cH^p=\cH(K^p)^{\sph}$\big). We summarize some (well-known) properties of $\cM$ and $\cS$ in the following theorem.
\begin{theorem}\label{thm: clp-eac}
  (1) The coherent sheaf $\cM$ is Cohen-Macauly over $\cS$.

  (2) The rigid space $\cS$ is equidimensional of dimension $2$, and the points of non-critical slope are Zariski-dense in $\cS$, and accumulate at points $(\chi,\lambda)$ with $\chi$ locally algebraic.

  (3) The rigid space $\cS$ is reduced.
\end{theorem}
\begin{proof}
  (1) follows from Proposition \ref{prop: eigna}  and the argument in the proof of \cite[Lem. 3.8]{BHS2}. Since $\widehat{T}_0$ is equidimensional of dimension $2$, by \cite[Prop. 6.4.2]{Che} and Proposition \ref{prop: eigna}, $\cS$ is also equidimensional of dimension $2$. The density and the accumulation property of the points of non-critical slope follow from standard arguments as in \cite[\S~6.4.5]{Che}. Finally, since the action of $\cH^p$ on $\widetilde{H}^1_{\et,c}(K^p,E)_{\overline{\rho}}^{\lalg}$ is semi-simple, (3) follows by the same argument as in \cite[Prop. 6.4]{Che05}.
\end{proof}
\begin{remark}\label{dens00}
 By  Theorem \ref{thm: clp-eac} (2) and Proposition \ref{prop: clp-cla}, the very classical points are Zariski-dense in $\cS$.
\end{remark}
Let $z=(\chi_z=\chi_{z,1}\otimes \chi_{z,2},\lambda_z)\in \cS$. Recall that one can associate (by the theory of pseudo-characters) to  $z$ a semi-simple continuous representation $\rho_z: \Gal_{\Q}\ra \GL_2(k_z)$ (where $k_z$ denotes the residue field of $z$, which is a finite extension of $E$) satisfying that
\begin{enumerate}
\item the mod $p$ reduction of $\rho_z$ is isomorphic to $\overline{\rho}$,
\item the restriction $\rho_{z,\ell}:=\rho_z|_{\Gal_{\Q_{\ell}}}$ is unramified for all $\ell \notin \Sigma(K^p)$, and
  \begin{equation*}
    \rho_z(\Frob_{\ell}^{-2})-\lambda_z(T_{\ell})\rho_z(\Frob_{\ell}^{-1})+\ell \lambda_z(S_{\ell})=0.
  \end{equation*}
\end{enumerate}
Note that the first property together with our assumption on $\overline{\rho}$ actually imply that $\rho_z$ is absolutely irreducible. Note also that $\rho_z$ is determined by the second property by Chebotarev's density theorem.
Let $z=(\chi_z, \lambda_z)\in \cS$ be such that $\wt(z):=\wt(\chi_{z,1})-\wt(\chi_{z,2})\in \Z_{\geq 0}$, we put
\begin{equation*}\chi_z^c:=\chi_z(x^{-\wt(z)-1} \otimes x^{\wt(z)+1}).\end{equation*}
 If $z^c:=(\chi_z^c, \lambda_z)\in \cS$, we call $z^c$ a \emph{companion point} of $z$. Suppose $z$ is classical, then we call $z$ \emph{critical} if $z$ admits a companion point. By \cite[Prop. 4.5.5]{Em1} (and the proof), we have (noting that the bad points in \emph{loc. cit.} are exactly the points admitting companion points in our terminology, see \cite[Def. 4.5.4]{Em1})
\begin{proposition}\label{sloNC}
If $z$ is of non-critical slope, then $z$ is not critical.
\end{proposition}
\begin{theorem}[$\text{\cite{BE}}$]\label{ThmBE}The point $z$ is critical if and only if the Galois representation $\rho_{z,p}:=\rho_z|_{\Gal_{\Q_p}}$ splits.
\end{theorem}

We recall the relation between Emerton's eigensurface $\cS$ and Coleman-Mazur eigencurve.
Let $\cW$ denote the rigid space parameterizing locally analytic characters of $\Z_p^{\times}$, the decomposition
\begin{equation}\label{decomp0}
  T(\Q_p)\cong \begin{pmatrix}
    \Z_p^{\times} & 0 \\ 0 & 1
  \end{pmatrix}\times \begin{pmatrix}
    1 & 0 \\ 0 & \Z_p^{\times}
  \end{pmatrix}\times \begin{pmatrix}
    p & 0 \\ 0 & 1
  \end{pmatrix}^{\Z} \times \begin{pmatrix}
    1 & 0 \\ 0 & p
  \end{pmatrix}^{\Z}
\end{equation}
induces $\widehat{T} \cong \cW_+\times \cW_-\times \bG_m \times \bG_m$ with $\cW_{\pm}\cong \cW$. The trivial character on $\cW_+$ induces an injection
\begin{equation*}
  \cW_-\times \bG_m \times \bG_m\hooklongrightarrow \widehat{T}.
\end{equation*}
Denote by $\cC$ the pull-back of $\cS$ over $\cW_-\times \bG_m \times \bG_m$, by $\kappa$ the induced map $\cC \ra \cW_-\cong \cW$. We still use $\cM$ to denote the pull-back of the $\co(\cS)$-module $\cM$ over $\cC$. We have
\begin{equation}\label{curve0}
  \cM(\cC)\cong \big(J_B(\widetilde{H}^1_{\et,c}(K^p,E)_{\overline{\rho}}^{\an})^{T_1}\big)^{\vee}_b
\end{equation}
where $T_1= \begin{pmatrix}
    \Z_p^{\times} & 0 \\ 0 & 1
  \end{pmatrix}$. By  \cite[Prop. 4.4.6]{Em1}, we have $\cS\cong \cC\times \cW$. Moreover, one can show that Proposition \ref{prop: eigna}, Theorem \ref{thm: clp-eac} and Remark \ref{dens00} also hold for $(\cC, \cM)$ (except that $\cC$ is equidimensional of dimension $1$). By \cite[Prop. 4.4.2]{Em1}, we have an explicit relation  between the classical points of $\cC$ and those of Coleman-Mazur eigencurve (constructed from the finite slope overconvergent modular forms). Using the reducedness of $\cC$, the density of (very) classical points and the same argument as in \cite[Prop. 7.2.8]{BCh}, one can deduce that $\cC$ is isomorphic to the corresponding Coleman-Mazur eigencurve. Finally we have by \cite[Prop. 4.5.5]{Em1}:
\begin{proposition}\label{prop: clp-neigh}
Let $z=(\chi_{z,1}\otimes \chi_{z,2},\lambda_z)\in \cC$, then there exists an admissible open $U\subset \cS$ containing $z$ such that any point in $U\setminus \{z\}$ does not admit companion points in $\cS$.
\end{proposition}
\section{Adjunctions}\label{sec: clp-1.3}
In this section, we study some adjunction properties of the Jacquet-Emerton functor.
Let $z=(\chi_z=\chi_{z,1}\otimes \chi_{z,2},\lambda_z)$ be an $E$-point of $\cC$.
By Proposition \ref{prop: eigna} (with $\cS$ replaced by $\cC$ as discussed  below Theorem \ref{ThmBE}) and Proposition \ref{prop: clp-neigh}, there exists an affinoid neighborhood $\cU_0$ of $z$ in $\cC$ such that
\begin{itemize}
\item $\kappa: \cU_0 \ra \kappa(\cU_0)$ is a finite morphism of affinoids;
\item $\cM(\cU_0)$ is finitely generated locally free over $\co(\kappa(\cU_0))$;
\item any $z'\in \cU_0(\overline{E})\setminus \{z\}$ does not admit companion points;
\item $\kappa^{-1}(\kappa(z))^{\red}=\{z\}$.
\end{itemize}
Let $\cU:=\kappa^{-1}(\cV)$ with $\cV$ an admissible strictly quasi-Stein  neighborhood of $\kappa(z)$ in $\kappa(\cU_0)$ (cf. \cite[Def. 2.1.17 (iv)]{Em04}. Thus $\cU$ also satisfies the above listed properties.  Since $\cM(\cU)$ is a finitely generated locally free $\co(\cV)$-module and $\cV$ is strictly quasi-Stein, we have that $\cM(\cU)^{\vee}_b$ is an \emph{allowable} locally analytic representation of $T(\Q_p)$ in the sense of \cite[Def. 0.11]{Em2} (e.g. using similar arguments as in \cite[Ex. 6.3.15 (ii)]{Ding}). Note also $\cM(\cU)^{\vee}_b$ is naturally equipped with a continuous action of $\cH^p$.
The restriction maps  $\cM(\cS) \ra \cM(\cC)\ra \cM(\cU)$ induce $\cH^p$-invariant morphisms of locally analytic $T(\Q_p)$-representations:
\begin{equation}\label{cplf: inja}\cM(\cU)^{\vee}_b \lra \cM(\cC)^{\vee}_b\lra \cM(\cS)^{\vee}_b\cong J_B\big(\widetilde{H}^1_{\et,c}(K^p,E)_{\overline{\rho}}^{\an}\big).
\end{equation}
When $z$ does not have companion points (hence all points in $\cU$ do not admit companion points by assumption), one can prove as in \cite[Lem. 4.5.12]{Em1} that the composition in   (\ref{cplf: inja}) is \emph{balanced} in the sense of  \cite[Def. 0.8]{Em2} and induces an $\cH^p$-invariant morphism of locally analytic $\GL_2(\Q_p)$-representations
\begin{equation}\label{equ: clp-seB}
  \big(\Ind_{\overline{B}(\Q_p)}^{\GL_2(\Q_p)} \cM(\cU)^{\vee}_b\otimes_E \delta_B^{-1}\big)^{\an} \lra \widetilde{H}^1_{\et,c}(K^p,E)_{\overline{\rho}}^{\an},
\end{equation}
where $\delta_B:=\unr(p^{-1})\otimes \unr(p)$ is the modulus character of $B(\Q_p)$ (which we also view as a character of $\overline{B}(\Q_p)$ via $\overline{B}(\Q_p)\twoheadrightarrow T(\Q_p)$). The morphism (\ref{equ: clp-seB}) plays a crucial role in Emerton's construction of two-variable $p$-adic $L$-functions  (cf. \cite[Thm. 4.5.7]{Em1}). However, if $z$ admits a companion point (e.g. if $z$ is critical), then (\ref{cplf: inja}) does not (directly) induce a such morphism (e.g. see Proposition \ref{prop: clp-ntr} below). The main result in this section is to establish a similar adjunction result in the locus of such  $z$.

Suppose henceforward $k:=\wt(\chi_{z,1})-\wt(\chi_{z,2})\in \Z_{\geq 0}$ (note $\wt(\chi_{z,1})=0$ by (\ref{curve0})), and let $h=\begin{pmatrix}
  1 & 0 \\ 0 & -1
\end{pmatrix}\in \ft$.  We have a natural injection $\ft\hookrightarrow  \co(\widehat{T}_0)$, and we view hence $h$ as an element in $\co(\widehat{T}_0)$. By (\ref{curve0}), the $\ft$-action on $\cM(\cU)^{\vee}_b$ factors through the projection $\ft=\begin{pmatrix} \Q_p & 0 \\ 0 & \Q_p\end{pmatrix} \twoheadrightarrow \Q_p$ onto the bottom right factor. Shrinking $\cV$ (and hence $\cU$), we assume $\kappa(z)$ is the only point in $\cV$ such that the associated character of $\Z_p^{\times}$ is of weight $-k$ (noting $\kappa(z)=\chi_{z,2}$).
 Consider $\cM(\cU)^{\vee}_b[h=k]$ which is a finite dimensional $E$-vector space, and is a  $T(\Q_p)\times \cH^p$-stable subspace of $\cM(\cU)^{\vee}_b$. By assumption, we have in fact: 
\begin{equation}\label{equ: clp-fiber}
  \cM(\cU)^{\vee}_b[h=k]=J_B\big(\widetilde{H}^1_{\et,c}(K^p,E)_{\overline{\rho}}^{\an}\big)[\ft=d\chi_z]\{T(\Q_p)=\chi_z, \cH^p=\lambda_z\}.
\end{equation}
We call a vector $v\in J_B\big(\widetilde{H}^1_{\et,c}(K^p,E)_{\overline{\rho}}^{\an}\big)$ \emph{classical} if $v$ lies in $J_B\big(\widetilde{H}^1_{\et,c}(K^p,E)_{\overline{\rho}}^{\lalg}\big)$ (compare with Definition \ref{class00}), and we call a vector $v\in \cM(\cU)^{\vee}_b$ classical if it is sent to  a classical vector in $J_B\big(\widetilde{H}^1_{\et,c}(K^p,E)_{\overline{\rho}}^{\an}\big)$ via the map (\ref{cplf: inja}).

\begin{lemma}\label{lem: clp-ber}
Let $\gamma \in \End_{T(\Q_p)}\big(\cM(\cU)^{\vee}_b\big)$, the composition
\begin{equation}\label{equ: clp-wUe}
  \cM(\cU)^{\vee}_b \xlongrightarrow{\gamma} \cM(\cU)^{\vee}_b \xlongrightarrow{\text{(\ref{cplf: inja})}} J_B\big(\widetilde{H}^1_{\et,c}(K^p,E)_{\overline{\rho}}^{\an}\big)
\end{equation}
is balanced  if and only if any vector in $\gamma(\cM(\cU)^{\vee}_b[h=k])$ is classical.
\end{lemma}
\begin{proof}
By definition and the same argument as in the proof of \cite[Lem. 4.5.12]{Em1}, (\ref{equ: clp-wUe}) is balanced if and only if for any $k'\in \Z_{\geq 0}$, the image of the composition
\begin{equation}\label{equ: clp-oge}
    \cM(\cU)^{\vee}_b[h=k'] \xlongrightarrow{\text{(\ref{equ: clp-wUe})}} J_B\big(\widetilde{H}^1_{\et,c}(K^p,E)_{\overline{\rho}}^{\an}\big)   \longrightarrow (\widetilde{H}^1_{\et,c}(K^p,E)_{\overline{\rho}}^{\an})^{N_0} \hooklongrightarrow \widetilde{H}^1_{\et,c}(K^p,E)_{\overline{\rho}}^{\an}
\end{equation}
is annihilated by the operator $X_-^{k'+1}$ where $X_-:=\begin{pmatrix} 0 & 0 \\ 1 & 0 \end{pmatrix}\in \text{U}(\ug)$ naturally acts on the locally analytic representation $\widetilde{H}^1_{\et,c}(K^p,E)_{\overline{\rho}}^{\an}$, $N_0=N(\Z_p)=\begin{pmatrix}
  1 & \Z_p \\ 0 & 1
\end{pmatrix}$, and where the second map in  (\ref{equ: clp-oge}) is the ``canonical lifting" of \cite[(3.4.8)]{Em2} (with respect to $N_0$),  equivariant under the action of
\begin{equation*}T(\Q_p)^+:=\bigg\{\begin{pmatrix}
  a & 0 \\ 0 & d
\end{pmatrix}\in T(\Q_p)\ |\ a/d \in \Z_p\setminus \{0\}\bigg\}.
\end{equation*}

One can show as in the proof of \cite[Lem. 4.5.12]{Em1} that if $k'\neq k$, then any vector $v$ in the image of the following composition (which is obtained in the same way as (\ref{equ: clp-oge}) replacing (\ref{equ: clp-wUe}) by (\ref{cplf: inja}))
\begin{equation}\label{equ: hk'}
     \cM(\cU)^{\vee}_b[h=k'] \xlongrightarrow{\text{(\ref{cplf: inja})}} J_B\big(\widetilde{H}^1_{\et,c}(K^p,E)_{\overline{\rho}}^{\an}\big) \longrightarrow \widetilde{H}^1_{\et,c}(K^p,E)_{\overline{\rho}}^{\an}
\end{equation}
is annihilated by $X_-^{k'+1}$. Indeed, using the fact that  $\cM(\cU)^{\vee}_b[h=k']$ is finite dimensional (and $T(\Q_p) \times \cH^p$-stable), we can assume without loss of generality that $v$ is a $(\chi',\lambda')$-eigenvector for $T(\Q_p)^+\times \cH^p$ with $(\chi',\lambda')\in \cU$. In this case, by \cite[Lem. 7.3.15]{Ding} (which is an easy variation of \cite[Prop. 4.4.4]{Em11}), $X_-^{k'+1} \cdot v \in \big(\widetilde{H}^1_{\et,c}(K^p,E)_{\overline{\rho}}^{\an}\big)^{N_0}$ is a generalized $((\chi')^c, \lambda')$-eigenvector for $T(\Q_p)^+\times \cH^p$. If $X_-^{k'+1} \cdot v \neq 0$, we deduce $\big(\widetilde{H}^1_{\et,c}(K^p,E)_{\overline{\rho}}^{\an}\big)^{N_0}\{T(\Q_p)^+=(\chi')^c, \cH^p=\lambda'\}\neq 0$, and hence (see \cite[Prop. 3.4.9]{Em11} for the first equality):
\begin{equation*}
  J_B\big(\widetilde{H}^1_{\et,c}(K^p,E)_{\overline{\rho}}^{\an}\big)[T(\Q_p)=(\chi')^c, \cH^p=\lambda']
  =\big(\widetilde{H}^1_{\et,c}(K^p,E)_{\overline{\rho}}^{\an}\big)^{N_0}[T(\Q_p)^+=(\chi')^c, \cH^p=\lambda']\neq 0.
\end{equation*}
But this implies $((\chi')^c, \lambda')\in \cS$, contradicting the fact $(\chi',\lambda')$ does not admit any companion point. Since (\ref{equ: clp-oge}) factors through (\ref{equ: hk'}), we deduce the image of  (\ref{equ: clp-oge}) for $k'\neq k$ is also annihilated by $X_-^{k'+1}$.

Now it suffices  to show that for $k'=k$ and a vector $v$ in the image of (\ref{equ: clp-oge}),  $X_-^{k+1} \cdot v=0$ if and only $v$ is classical. Using the fact $v$ is fixed by $N_0$, $\ft v=d\chi_z v$ and the highest weight theory, it is not difficult to see that that followings are equivalent
\begin{itemize}
\item $v$ is classical,
\item $\text{U}(\ug) v\cong (\Sym^{k} E^2)^{\vee}$,
\item $X_-^{k+1} \cdot v=0$.
\end{itemize}
The lemma then follows.
\end{proof}
Recall that in \cite[(2.8)]{Em2}, Emerton introduced a closed $\GL_2(\Q_p)$-subrepresentation \begin{equation}\label{equ: cpl-IPG}I_{\overline{B}(\Q_p)}^{\GL_2(\Q_p)}\big(\cM(\cU)^{\vee}_b\otimes_E \delta_B^{-1}\big)\hooklongrightarrow \big(\Ind_{\overline{B}(\Q_p)}^{\GL_2(\Q_p)} \cM(\cU)^{\vee}_b\otimes_E \delta_B^{-1}\big)^{\an},
 \end{equation}
which, by \cite[Lem. 2.8.3]{Em2},  is the  closed $\GL_2(\Q_p)$-subrepresentation of $$\big(\Ind_{\overline{B}(\Q_p)}^{\GL_2(\Q_p)} \cM(\cU)^{\vee}_b\otimes_E \delta_B^{-1}\big)^{\an}$$ generated by $\cM(\cU)^{\vee}_b$ via the following composition (see \cite[Lem. 0.3]{Em2} for the first map, and \cite[(3.4.8)]{Em11} for the second map)
\begin{equation*}
  \cM(\cU)^{\vee}_b \hooklongrightarrow J_B\big(\big(\Ind_{\overline{B}(\Q_p)}^{\GL_2(\Q_p)} \cM(\cU)^{\vee}_b\otimes_E \delta_B^{-1}\big)^{\an}\big) \longrightarrow \big(\Ind_{\overline{B}(\Q_p)}^{\GL_2(\Q_p)} \cM(\cU)^{\vee}_b\otimes_E \delta_B^{-1}\big)^{\an}.
\end{equation*}
However, by the assumption that $\cM(\cU)$ is locally free over $\co(\kappa(\cU))$, one can  prove as in \cite[Lem. 4.5.12 (ii)]{Em1} that (\ref{equ: cpl-IPG}) is an isomorphism.
The following theorem thus follows from Emerton's adjunction formula \cite[Thm. 0.13]{Em2} combined with Lemma \ref{lem: clp-ber}.
\begin{theorem}\label{thm: clp-rps}
  Keep the notation of Lemma \ref{lem: clp-ber}, and suppose $\gamma$ commutes with $\cH^p$ (so that(\ref{equ: clp-wUe}) is $\cH^p$-equivariant), then the followings are equivalent:

  (i) there exists an $\cH^p$-equivariant morphism of locally analytic $\GL_2(\Q_p)$-representations
  \begin{equation*}
    \big(\Ind_{\overline{B}(\Q_p)}^{\GL_2(\Q_p)} \cM(\cU)^{\vee}_b\otimes_E \delta_B^{-1}\big)^{\an}\lra \widetilde{H}^1_{\et,c}(K^p,E)^{\an}_{\overline{\rho}}
  \end{equation*}
  such that the induced morphism
  \begin{equation}\label{Adj002}
    \cM(\cU)^{\vee}_b \hooklongrightarrow J_B\big(\big(\Ind_{\overline{B}(\Q_p)}^{\GL_2(\Q_p)} \cM(\cU)^{\vee}_b\otimes_E \delta_B^{-1}\big)^{\an}\big)\lra J_B\big(\widetilde{H}^1_{\et,c}(K^p,E)^{\an}_{\overline{\rho}}\big)
  \end{equation}
  is equal to (\ref{equ: clp-wUe});

  (ii) any vector in $\gamma(\cM(\cU)^{\vee}_b[h=k])$ is classical.
\end{theorem}
Let  $\cM(\cU)^{\vee}_b[h=k]^{\cl}  \subseteq \cM(\cU)^{\vee}_b[h=k]$ be the subspace of classical vectors. Recall we have $\cM(\cU)^{\vee}_b[h=k]\cong (\cM(\cU)/(h-k))^{\vee}$ on which the $\co(\cU)$-action  is determined by the action of $T(\Q_p)\times \cH^p$. By definition (see also (\ref{equ: clp-fiber})), one easily verifies that $\cM(\cU)^{\vee}_b[h=k]^{\cl} $ is stable under the action of $T(\Q_p)$ and $\cH^p$. Hence $\cM(\cU)^{\vee}_b[h=k]^{\cl}$ is stable under the $\co(\cU)$-action.
Put
\begin{equation}\label{Iz}
  \cI_z:=\{\gamma \in \co(\cU)\ |\ \gamma(\cM(\cU)^{\vee}_b[h=k]) \subseteq \cM(\cU)^{\vee}_b[h=k]^{\cl}\}.
\end{equation}
We see $\cI_z$ is either an ideal of $\co(\cU)$ of $\cI_z=\co(\cU)$. It is also clear that $h-k\in \cI_z$.
\begin{proposition}\label{prop: clp-ntr}Suppose $\chi_z$ is a product of an algebraic character and an unramified character, then
  $\cI_z=\co(\cU)$ if and only if $z$ does not admit companion points.
\end{proposition}
\begin{proof}
  If $z$ does not have companion points, by the same argument as in the proof of Lemma \ref{lem: clp-ber}, we see any vector in $\cM(\cU)^{\vee}_b[h=k]$ is classical, hence $\cI_z=\co(\cU)$.

  Suppose now $\cI_z=\co(\cU)$. By Theorem \ref{thm: clp-rps},  any vector in $\cM(\cU)^{\vee}_b[h=k]$ is classical. By (\ref{equ: clp-fiber}),  any vector in $J_B\big(\widetilde{H}^1_{\et,c}(K^p,E)_{\overline{\rho}}^{\an}\big)[h=k]\{T(\Q_p)=\chi_z, \cH^p=\lambda_z\}$ is classical (in particular, the point $z$ is classical). Assume first $\chi_z=\chi_{z,1} \otimes \chi_{z,2}$ satisfies $\chi_{z,1} \chi_{z,2}^{-1} \neq \unr(p^{-1}) x^k$. In this case, by similar (and easier) argument as in the proof of \cite[Thm. 7]{Ding3} (using Chenevier's method \cite[\S~4.4]{Che11}), we can deduce that $\cC$ is \'etale over $\cW$ at the point $z$. However, by \cite[Thm. 1.1]{Bergd0}, if $z$ is critical, then $\cC$ will be ramified over $\cW$ at $z$, that leads to a contradiction.
  Now assume $\chi_{z,1} \chi_{z,2}^{-1} = \unr(p^{-1}) x^k$. Together with (\ref{center}), we easily deduce $\val_p(\chi_{z,1}(p))=(k-1)/2$. Since $z\in \cC$, the character $\chi_{z,1}$ is actually smooth.  By Definition \ref{class00} (2), the point $z$ is of non-critical slope, and hence non-critical by Proposition \ref{sloNC}. This concludes the proof.
\end{proof}

Under the $\pi_0$-action, the coherent sheaf $\cM$ naturally decomposes into $\cM^+\oplus \cM^-$. We deduce that  $\cM^{\pm}(\cU)$ are both locally free over $\co(\kappa(\cU))$. Moreover by considering their fibers at classical points, we can see both of them are non-zero. The precedent results also hold for $\{\cM^{\pm}(\cU), \widetilde{H}^1_{\et,c}(K^p,E)_{\overline{\rho}}^{\an,\pm}\}$. In particular, we can define ideals $\cI_z^{\pm}$ in a similar way, and we can prove that the same statement in Proposition \ref{prop: clp-ntr} holds with $\cI_z$ replaced by  $\cI_z^{\pm}$.
\section{Two-variable $p$-adic $L$-functions}
We use the results in \S~\ref{sec: clp-1.3} to construct two-variable $p$-adic $L$-functions in a neighborhood  of critical points in $\cC$.
\subsection{Constructions}\label{sec: 21}
Let $N>1$, $p\nmid N$, and let
\begin{equation*}
  K^p=\bigg\{\begin{pmatrix} a & b \\ c & d \end{pmatrix}\in \GL_2(\hat{\Z}^p) \ \bigg| \ c\equiv 0 \pmod{N},\  d\equiv 1 \pmod N\bigg\}.
\end{equation*}
In this section, we let $\cH^p$ denote the $\co_E$-algebra generated by $T_{\ell}$, $S_{\ell}$ for $\ell \nmid N$, $\ell \neq p$, and the diamond operators $\langle a \rangle$ for $a\in (\Z/N\Z)^{*}$. Let $k\in \Z_{\geq 0}$, $f$ be a newform of weight $k+2$, level $\Gamma_1(N)$ over $E$, i.e. $f$ is an eigenform for $\cH^p$ and $T_p$, $S_p$, such that there does not exist an  eigenform (for $\cH^p$) of weight $k+2$, of level $\Gamma_1(N')$ with $N'$ is a proper divisor of $N$, having the same eigenvalue as $f$ for the operators in $\cH^p$. Let $\rho_f$ be the associated $\Gal_{\Q}$-representation over $E$ (enlarge $E$ if necessary), and we assume that the mod $\varpi_E$ reduction of $\rho_f$ is isomorphic to $\overline{\rho}$. Let $\cC$, $\cM$ be as in \S~\ref{sec: clp-1.2}.  Let $\epsilon: (Z/NZ)^{\times} \ra E^{\times}$ be the character with $\langle a\rangle f=\epsilon(a) f$.  Let $a_p, b_p\in E$ (enlarge $E$ if necessary) be the eigenvalues of $T_p$, $S_p$ of $f$ respectively. Let $\alpha$ be a root of $X^2-a_p X+pb_p=0$, and put
\begin{equation*}
  f_{\alpha}(z)=f(z)-\frac{pb_p}{\alpha} f(pz),\\
\end{equation*}
which is a modular form of weight $k+2$, of level $\Gamma_1(N)\cap \Gamma_0(p)$, and is an  eigenform of the same eigenvalues of $f$ for $\cH^p$,  an eigenform for $U_p$ of eigenvalues $\alpha$. The form $f_{\alpha}$ is referred to as a \emph{refinement} (or a \emph{$p$-stabilization}) of $f$. By \cite[Prop. 4.4.2]{Em1}, $f_{\alpha}$ corresponds to  a classical point $z_{\alpha}=(\chi_{z_{\alpha}},\lambda_f)\in \cC\subset \cS$ where $\lambda_f$ denotes the system of eigenvalues of $f$ for operators in $\cH^p$, and where
\begin{equation*}
  \chi_{z_{\alpha}}=\unr(\alpha/p)\otimes x^{-k} \unr(pb_p/\alpha).\\
\end{equation*}
In this section, we construct two-variable $p$-adic $L$-functions in a neighborhood of $z_{\alpha}$ (especially in the case where $z_{\alpha}$ is critical), via Emerton's method \cite{Em05}\cite{Em1}.
\begin{proposition}\label{locfree}
  $\cM^{\pm}$ are locally free of rank $1$ in a neighborhood of $z_{\alpha}$ in $\cC$.
\end{proposition}
\begin{proof}
Suppose first $\alpha\neq pb_p/\alpha$ (i.e. $X^2-a_p X+pb_p=0$ has two distinct roots). By \cite[\S~2.3]{Bel}, the eigencurve $\cC$ is smooth at the classical point $z_{\alpha}$. Since $\cM$ is Cohen-Macauly (see Theorem \ref{thm: clp-eac} (1) and the discussion below (\ref{curve0})), by \cite[Cor. 17.3.5 (i)]{EGAiv1}, $\cM$ is locally free around $z_{\alpha}$. Hence $\cM^{\pm}$ are also locally free around $z_{\alpha}$. For any classical point $z'=(\chi_{z'},\lambda_{z'})$ of non-critical slope in $\cC$, we have
\begin{multline}\label{multi00}
  \dim_{k_{z'}} (z')^* \cM^{\pm}= \dim_{k_{z'}} J_B\big(\widetilde{H}^1_{\et,c}(K^p,E)_{\overline{\rho}}^{\an,\pm}\big)[T(\Q_p)=\chi_{z'},\cH^p=\lambda_{z'}]\\
  =\dim_{k_{z'}} J_B\big(\widetilde{H}^1_{\et,c}(K^p,E)_{\overline{\rho}}^{\lalg,\pm}\big)[T(\Q_p)=\chi_{z'},\cH^p=\lambda_{z'}]=1.
\end{multline}
where the first equality follows from (\ref{fib0}), the second follows from Proposition \ref{prop: clp-cla}, and the last equality follows from the multiplicity one result (see \cite[Prop. 4.4.18]{Em1}).
Since such points accumulate at $z_{\alpha}$ (see Theorem \ref{thm: clp-eac} (2) and the discussion below (\ref{curve0})), we deduce  that $\cM^{\pm}$ are locally free of rank $1$ in a neighborhood of $z_{\alpha}$ in $\cC$. If $\alpha=pb_p/\alpha$, as in the proof of Proposition \ref{prop: clp-ntr}, we know $z_{\alpha}$ is of non-critical slope. The proposition in this case then follows from \cite[Prop. 4.4.20]{Em1} (and the proof).
\end{proof}Let $\cU$ be a neighborhood of $z_{\alpha}$ in $\cC$ as in \S~\ref{sec: clp-1.3}.  If $\alpha\neq pb_p/\alpha$, we shrink $\cU$ such that the maximal ideal $\fm_{z_{\alpha}}$ associated to $z_{\alpha}$ is generated by one element $r_{z_{\alpha}}\in \co(\cU)$ (using the fact $\cU$ is one-dimensional and smooth at the point $z_{\alpha}$). We let $e$ be the ramification degree of $\cC$ over $\cW$ at $z_{\alpha}$ in this case. If $\alpha=pb_p/\alpha$, we put $r_{z_{\alpha}}=1$ and $e=1$.
\begin{theorem}\label{adj003}
  The composition (that is $\cH^p \times T(\Q_p)$-equivariant)  \begin{equation}\label{equ: clp-cEp}
  \cM^{\pm}(\cU)^{\vee}_b \xlongrightarrow{r_{z_{\alpha}}^{e-1}} \cM^{\pm}(\cU)^{\vee}_b \lra J_B\big(\widetilde{H}^1_{\et,c}(K^p,E)_{\overline{\rho}}^{\an}\big)
\end{equation}
induces  an $\cH^p$-equivariant morphism of locally analytic representations of $\GL_2(\Q_p)$
\begin{equation}\label{equ: clp-ppA}
  \iota^{\pm}: \big(\Ind_{\overline{B}(\Q_p)}^{\GL_2(\Q_p)} \cM^{\pm}(\cU)^{\vee}_b\otimes_E \delta_B^{-1}\big)^{\an} \lra \widetilde{H}^1_{\et,c}(K^p,E)_{\overline{\rho}}^{\an,\pm}.
\end{equation}
\end{theorem}
\begin{proof}
  If $\alpha=pb_p/\alpha$, then $z_{\alpha}$ is of non-critical slope (see the proof of Proposition \ref{prop: clp-ntr}), and hence non-critical by Proposition \ref{sloNC}. By Proposition \ref{prop: clp-ntr} (and the discussion that follows), we see $\cI_z^{\pm}=\co(\cU)$. The theorem then follows from Theorem \ref{thm: clp-rps}.

  We suppose $\alpha\neq pb_p/\alpha$. We prove first
    \begin{equation}\label{chaofcl}
\cM^{\pm}(\cU)^{\vee}_b[r_{z_{\alpha}}]=\cM^{\pm}(\cU)_b^{\vee}[h=k]^{\cl}.
  \end{equation}
We have
\begin{multline}\label{ss00}\cM^{\pm}(\cU)_b^{\vee}[h=k]^{\cl}=J_B\big(\widetilde{H}^1_{\et,c}(K^p,E)_{\overline{\rho}}^{\lalg,\pm}\big)[\ft=d\chi_z]\{T(\Q_p)=\chi_z,\cH^p=\lambda_z\} \\=J_B\big(\widetilde{H}^1_{\et,c}(K^p,E)_{\overline{\rho}}^{\lalg,\pm}\big)[T(\Q_p)=\chi_z,\cH^p=\lambda_z]
\end{multline}where the first equality follows from (\ref{equ: clp-fiber}), and the second from the assumption $\alpha\neq pb_p/\alpha$ and the semi-simplicity of the $\cH^p$-action on $\widetilde{H}^1_{\et,c}(K^p,E)_{\overline{\rho}}^{\lalg}$ (using the classical fact that the $\cH^p$-action on the right hand side of (\ref{localg}) is semi-simple).
By (the proof of) \cite[Prop. 4.4.18]{Em1}, the $E$-vector space in (\ref{ss00}) is one dimensional. It is clear that we have an injection
\begin{multline}\label{inja}J_B\big(\widetilde{H}^1_{\et,c}(K^p,E)_{\overline{\rho}}^{\lalg,\pm}\big)[T(\Q_p)=\chi_z,\cH^p=\lambda_z] \\ \hooklongrightarrow J_B\big(\widetilde{H}^1_{\et,c}(K^p,E)_{\overline{\rho}}^{\an,\pm}\big)[T(\Q_p)=\chi_z,\cH^p=\lambda_z]\cong (z_{\alpha}^* \cM^{\pm})^{\vee}=\cM^{\pm}(\cU)^{\vee}_b[r_{z_{\alpha}}].\end{multline}
By Proposition \ref{locfree}, $\dim_E \cM^{\pm}(\cU)^{\vee}_b[r_{z_{\alpha}}]=1$. We see (\ref{inja}) is actually bijective. Combined with (\ref{ss00}), (\ref{chaofcl}) follows.

We have $(h-k)\co(\cU)=r_{z_{\alpha}}^e \co(\cU)$. By Proposition \ref{locfree}, we have $\cM^{\pm}(\cU)_b^{\vee}[h=k]\cong (\co(\cU)/(h-k))^{\vee}$ and by (\ref{chaofcl}) $\cM^{\pm}(\cU)_b^{\vee}[h=k]^{\cl}\cong (\co(\cU)/r_{z_{\alpha}})^{\vee}$. We deduce (see (\ref{Iz}) and the discussion at the end of \S~\ref{sec: clp-1.3})
\begin{equation}\label{Iz2}\cI_z=\cI_z^{\pm}=(\frac{h-k}{r_{z_\alpha}}) \co(\cU)=r_{z_{\alpha}}^{e-1} \co(\cU).
\end{equation}
The theorem follows then from Theorem \ref{thm: clp-rps}.
\end{proof}

We have $\cM^{\pm}(\cU)^{\vee}_b[r_{z_{\alpha}}]\cong \chi_{z_{\alpha}}$ as $T(\Q_p)$-representation, thus (\ref{equ: clp-ppA}) induces a morphism
\begin{equation}\label{equ: clp-adj}
  \big(\Ind_{\overline{B}(\Q_p)}^{\GL_2(\Q_p)} \chi_{z_{\alpha}} \delta_B^{-1}\big)^{\an} \lra \widetilde{H}^1_{\et,c}(K^p,E)_{\overline{\rho}}^{\an,\pm}[\cH^p=\lambda_f].
\end{equation}
The locally analytic representation $ (\Ind_{\overline{B}(\Q_p)}^{\GL_2(\Q_p)}\chi_{z_{\alpha}}\delta_B^{-1})^{\an}$ sits in a non-splitting  exact sequence (e.g. see \cite[Thm. 4.1]{Br})
\begin{multline}\label{equ: clp-exact}
 0 \lra (\Ind_{\overline{B}(\Q_p)}^{\GL_2(\Q_p)} \chi_{z_{\alpha}}^{\sm}\delta_B^{-1})^{\infty}\otimes_E (\Sym^k E^2)^{\vee} \\ \lra  (\Ind_{\overline{B}(\Q_p)}^{\GL_2(\Q_p)}\chi_{z_{\alpha}}\delta_B^{-1})^{\an} \lra  (\Ind_{\overline{B}(\Q_p)}^{\GL_2(\Q_p)} \chi_{z_{\alpha}}^c\delta_B^{-1})^{\an}\ra 0,
\end{multline}
where $\chi_{z_{\alpha}}^{\sm}:=\unr(\alpha/p)\otimes \unr(pb_p/\alpha)$ is the ``smooth" part of $\chi_{z_{\alpha}}$, $(\Ind -)^{\infty}$ denotes the smooth induction, and recall  $\chi_{z_{\alpha}}^c=\chi_{z_{\alpha}} (x^{1-k} \otimes x^{k-1})$.
\begin{proposition}\label{prop: clp-rest}
 The map (\ref{equ: clp-adj}) is non-zero, and its restriction on the locally algebraic vectors is zero if and only if $z_{\alpha}$ is critical.
\end{proposition}
\begin{proof}We first prove the second part of the proposition. Put \begin{equation*}I(\chi_{z_{\alpha}} \delta_B^{-1}):=(\Ind_{\overline{B}(\Q_p)}^{\GL_2(\Q_p)} \chi_{z_{\alpha}}^{\sm}\delta_B^{-1})^{\infty}\otimes_E (\Sym^k E^2)^{\vee}\end{equation*}
 which is in fact the locally algebraic subrepresentation of $\big(\Ind_{\overline{B}(\Q_p)}^{\GL_2(\Q_p)} \chi_{z_{\alpha}} \delta_B^{-1}\big)^{\an}$ (cf. (\ref{equ: clp-exact})). Applying the (left exact) Jacquet-Emerton functor to the composition
\begin{equation}
   I(\chi_{z_{\alpha}} \delta_B^{-1})\hooklongrightarrow \big(\Ind_{\overline{B}(\Q_p)}^{\GL_2(\Q_p)} \chi_{z_{\alpha}} \delta_B^{-1}\big)^{\an} \xlongrightarrow{\text{(\ref{equ: clp-adj})}} \widetilde{H}^1_{\et,c}(K^p,E)_{\overline{\rho}}^{\an,\pm}[\cH^p=\lambda_f],
\end{equation}
one gets an injection of locally analytic representations of $T(\Q_p)$ (see \cite[Lem. 0.3]{Em2}):
\begin{equation}\label{equ: clp-ggb}
  \chi_{z_{\alpha}}\hooklongrightarrow J_B\big(\widetilde{H}^1_{\et,c}(K^p,E)_{\overline{\rho}}^{\an,\pm}\big)[\cH^p=\lambda_f].
\end{equation}
We have actually a commutative diagram (see (\ref{Adj002}) for the bottom maps)
\begin{equation*}
  \begin{CD}
    \chi_{z_{\alpha}}@>>> J_B\big(\Ind_{\overline{B}(\Q_p)}^{\GL_2(\Q_p)} \chi_{z_{\alpha}} \delta_B^{-1}\big)^{\an} @>>> J_B\big(\widetilde{H}^1_{\et,c}(K^p,E)_{\overline{\rho}}^{\an,\pm}\big)[\cH^p=\lambda_f] \\@VVV @VVV @VVV\\
    \cM(\cU)^{\vee}_b @>>> J_B\big(\big(\Ind_{\overline{B}(\Q_p)}^{\GL_2(\Q_p)} \cM(\cU)^{\vee}_b\otimes_E \delta_B^{-1}\big)^{\an}\big) @>>> J_B\big(\widetilde{H}^1_{\et,c}(K^p,E)^{\an}_{\overline{\rho}}\big)
  \end{CD}
\end{equation*}
Using Theorem \ref{thm: clp-rps}, we see (\ref{equ: clp-ggb}) is equal to the composition
\begin{equation*}
  \chi_{z_{\alpha}}\cong \cM^{\pm}(\cU)^{\vee}_b[r_{z_{\alpha}}] \hooklongrightarrow \cM^{\pm}(\cU)^{\vee}_b \xlongrightarrow{\text{(\ref{equ: clp-cEp})}} J_B\big(\widetilde{H}^1_{\et,c}(K^p,E)_{\overline{\rho}}^{\an,\pm}\big).
\end{equation*}
It is  straightforward to see the above map is non-zero if and only if $e=1$, which is equivalent to that  $z_{\alpha}$ is non-critical by Proposition \ref{prop: clp-ntr}. The second part follows.

Now we prove (\ref{equ: clp-adj}) is non-zero. The non-critical case is an easy consequence of the second part proved above. Now assume $z_{\alpha}$ is critical (in particular $\alpha\neq pb_p/\alpha$). If (\ref{equ: clp-adj}) is zero, we see the morphism
\begin{equation*}
  (\Ind_{\overline{B}(\Q_p)}^{\GL_2(\Q_p)} \cM^{\pm}(\cU)^{\vee}_b \otimes_E \delta_B^{-1})^{\an} \lra \widetilde{H}^1_{\et,c}(K^p,E)_{\overline{\rho}}^{\an,\pm}
\end{equation*}
factors as \Big(noting the kernel of the first map of the following composition is $$ (\Ind_{\overline{B}(\Q_p)}^{\GL_2(\Q_p)} \cM^{\pm}(\cU)^{\vee}_b[r_{z_{\alpha}}] \otimes_E \delta_B^{-1})^{\an} \cong \big(\Ind_{\overline{B}(\Q_p)}^{\GL_2(\Q_p)} \chi_{z_{\alpha}} \delta_B^{-1}\big)^{\an}\Big):$$
\begin{equation*}
  (\Ind_{\overline{B}(\Q_p)}^{\GL_2(\Q_p)} \cM^{\pm}(\cU)^{\vee}_b \otimes_E \delta_B^{-1})^{\an} \xlongrightarrow{r_{z_{\alpha}}} (\Ind_{\overline{B}(\Q_p)}^{\GL_2(\Q_p)} \cM^{\pm}(\cU)^{\vee}_b \otimes_E \delta_B^{-1})^{\an} \lra \widetilde{H}^1_{\et,c}(K^p,E)_{\overline{\rho}}^{\an,\pm},
\end{equation*}
where the second morphism is also $\GL_2(\Q_p)\times \cH^p$-equivariant.
Applying the Jacquet-Emerton functor, we deduce (\ref{equ: clp-cEp}) can factor as
\begin{equation}\label{comp11}
  \cM^{\pm}(\cU)^{\vee}_b  \xlongrightarrow{r_{z_{\alpha}}} \cM^{\pm}(\cU)^{\vee}_b \xlongrightarrow{\iota'} J_B\big(\widetilde{H}^1_{\et,c}(K^p,E)_{\overline{\rho}}^{\an,\pm}\big),
\end{equation}
where $\iota'$ is $T(\Q_p)\times \cH^p$-equivariant.
By \cite[Thm. 0.13]{Em2}, the map $\iota'$ is balanced. Using the same argument as in the proof of Lemma \ref{lem: clp-ber}, we see that any vector in $\iota'(\cM^{\pm}(\cU)^{\vee}_b[h=k])$ is classical.  The composition
\begin{equation}\label{pol0}r_{z_{\alpha}} \circ \iota': \cM^{\pm}(\cU)^{\vee}_b[h=k] \lra \cM^{\pm}(\cU)^{\vee}_b[h=k]
 \end{equation}is equal to $r_{z_{\alpha}}^{e-1}$ (since the two maps (\ref{comp11}) and (\ref{equ: clp-cEp}) are equal, and their restriction on $\cM^{\pm}(\cU)^{\vee}_b[h=k]$ is injective using (\ref{equ: clp-fiber})).  However, $\iota'(\cM^{\pm}(\cU)^{\vee}_b[h=k])=\cM^{\pm}(\cU)^{\vee}_b[h=k]^{\cl}=\cM^{\pm}(\cU)^{\vee}_b[r_{z_{\alpha}}]$ and hence the morphism (\ref{pol0}) is zero, which contradicts $r_{z_{\alpha}}^{e-1}\neq 0$ on $\cM^{\pm}(\cU)^{\vee}_b[h=k]\cong (\co(\cU)/r_{z_\alpha}^e)^{\vee}$. The proposition follows.
\end{proof}
\begin{remark}\label{rem: clp-classical} If $z_{\alpha}$ is not critical then (\ref{equ: clp-adj}) is injective, and its restriction on the locally algebraic subrepresentations is given by the local-global compatibility in classical local Langlands correspondence. 
If $z_{\alpha}$ is critical then (\ref{equ: clp-adj}) factors as
\begin{equation}\label{compan} \big(\Ind_{\overline{B}(\Q_p)}^{\GL_2(\Q_p)} \chi_{z_{\alpha}} \delta_B^{-1}\big)^{\an} \twoheadlongrightarrow(\Ind_{\overline{B}(\Q_p)}^{\GL_2(\Q_p)} \chi_{z_{\alpha}}^c\delta_B^{-1})^{\an}\hooklongrightarrow  \widetilde{H}^1_{\et,c}(K^p,E)_{\overline{\rho}}^{\an,\pm}[\cH^p=\lambda_f].
\end{equation}
The existence of the second injection was proved in \cite{BE}, which is an important fact on local-global compatibility in $p$-adic local Langlands program in critical case.
\end{remark}
Let $D^0$ be the $\co_E$-module of degree zero divisors in $\bP^1(\Q)$.
Recall that by \cite[Prop. 4.2]{Em05}, one has the following pairing which interpolates the classical modular symbols
\begin{equation*}
  \widetilde{H}^1_{\et,c}(K^p,\co_E)_{\overline{\rho}}\times D^0 \lra \co_E.
\end{equation*}Evaluating at $\{\infty\}-\{0\}$ gives a continuous $E$-linear map (of norm less than $1$)
\begin{equation}\label{equ: clp-cpa}
\{\infty\}-\{0\}:  \widetilde{H}^1_{\et,c}(K^p,E)_{\overline{\rho}} \lra E.
\end{equation}
We fix an isomorphism of $\co(\cU)$-modules $\cM^{\pm}(\cU)\cong \co(\cU)$. The following composition
\begin{multline}\label{padicL}
  \cC^{\la}(\Z_p^{\times},E)\widehat{\otimes}_E \co(\cU)^{\vee}_b \cong \cC^{\la}(\Z_p^{\times},\co(\cU)^{\vee}_b) \hooklongrightarrow  \big(\Ind_{\overline{B}(\Q_p)}^{\GL_2(\Q_p)} \co(\cU)^{\vee}_b\otimes_E \delta_B^{-1}\big)^{\an}
  \\ \xlongrightarrow{\sim}  \big(\Ind_{\overline{B}(\Q_p)}^{\GL_2(\Q_p)} \cM^{\pm}(\cU)^{\vee}_b\otimes_E \delta_B^{-1}\big)^{\an}
  \xlongrightarrow{\iota^{\pm}}\widetilde{H}^1_{\et,c}(K^p,E)_{\overline{\rho}}^{\pm, \an} \xlongrightarrow{\{\infty\}-\{0\}} E,
\end{multline}
thus gives a global section $L^{\pm}$ of $\cU \times\cW$, where the second map sends $F$ to $I(F)$ with $I(F)$ supported in $\overline{B}(\Q_p)N(\Z_p^{\times})$ with $N(\Z_p^{\times}):=\begin{pmatrix}1 & \Z_p^{\times} \\ 0 & 1 \end{pmatrix}$, and $I(F)\bigg(\begin{pmatrix} 1 & x \\ 0 & 1 \end{pmatrix}\bigg)=F(x)$ for $x\in \Z_p^{\times}$. The function $L^{\pm}$ was constructed in \cite[Thm. 4.5.7]{Em1} when $z_{\alpha}$ is of non-critical slope. We also remark that $L^{\pm}$ depends on the (fixed) isomorphism $\cM^{\pm}(\cU)\cong \co(\cU)$, and hence is naturally defined up to units in $\co(\cU)$.

Denote by $L^{\pm}(z_{\alpha},-): \cC^{\la}(\Z_p^{\times}, E) \ra E$ the evaluation of $L^{\pm}$ at $z_{\alpha}$ which is  a distribution of $\Z_p^{\times}$. Up to non-zero scalars, $L^{\pm}(z_{\alpha},-)$ equals the distribution given by the following composition
\begin{equation}\label{equ: clp-mea}
    \cC^{\la}(\Z_p^{\times}, E)\hooklongrightarrow  \big(\Ind_{\overline{B}(\Q_p)}^{\GL_2(\Q_p)} \chi_{z_{\alpha}} \delta_B^{-1}\big)^{\an} \xlongrightarrow{(\ref{equ: clp-adj})}  \widetilde{H}^1_{\et,c}(K^p,E)_{\overline{\rho}}^{\an,\pm} \xlongrightarrow{(\ref{equ: clp-cpa})} E,
  \end{equation}
  where the first map sends $F$ to $I(F)$ with $I(F)$ supported in $\overline{B}(\Q_p)N(\Z_p^{\times})$, and $I(F)\bigg(\begin{pmatrix} 1 & x \\ 0 & 1 \end{pmatrix}\bigg)=F(x)$ for $x\in \Z_p^{\times}$.
  The following proposition is due to Emerton.
\begin{proposition} \label{intpl}
  (1) The distribution $L^{\pm}(z_{\alpha},-)$ is $\alpha$-tempered (see \cite[Def. 3.12]{Em05} for the definition of $\alpha$-tempered distributions).

  (2) Suppose $z_{\alpha}$ is not critical, then we have up to non-zero scalars (independent of $\phi x^j$),
\begin{equation}\label{equ: clp-int}
  L^{\pm}(z_{\alpha},\phi x^j)=e_p(\alpha,\phi x^j)\frac{m^{j+1}}{(-2\pi i)^j}\frac{j!}{\tau(\phi^{-1})}\frac{L_{\infty}(f\phi^{-1},j+1)}{\Omega_f^{\pm}},
\end{equation}
where $\phi: \Z_p^{\times}\ra \Q_p^{\times}$ is a smooth character of conductor $p^v$ satisfying $\phi(-1)=\pm 1$, $j\in\{0, \cdots, k\}$,
\begin{equation*}
  e_p(\alpha,\phi x^j):=\frac{1}{\alpha^v}\big(1-\frac{\phi^{-1}(p)\epsilon(p) p^{k-j}}{\beta}\big)\big(1-\frac{\phi(p) p^j}{\beta}\big),
\end{equation*}
$\tau(\phi^{-1})$ is the Gauss sum of $\phi$, $L_{\infty}$ is the archimedean $L$-function and $\Omega_f^{\pm}$ are the two archimedean periods of $f$.
\end{proposition}
\begin{proof}
(1)  By \cite[Lem. 3.22]{Em05}, the composition
\begin{equation*}
  \cC^{\la}(\Z_p^{\times}, E)\hooklongrightarrow  \big(\Ind_{\overline{B}(\Q_p)}^{\GL_2(\Q_p)} \chi_{z_{\alpha}} \delta_B^{-1}\big)^{\an} \xlongrightarrow{(\ref{equ: clp-adj})}  \widetilde{H}^1_{\et,c}(K^p,E)_{\overline{\rho}}^{\pm}
\end{equation*}
is $\alpha$-tempered (in the sense of \cite[Def. 3.12]{Em05}).
Since (\ref{equ: clp-cpa}) sends $\widetilde{H}^1_{\et,c}(K^p,\co_E)_{\overline{\rho}}$ to $\co_E$, we deduce by \cite[Lem. 3.20]{Em05} that  $L^{\pm}(z_{\alpha},-)$
is $\alpha$-tempered.

(2) The interpolation result for the distribution (\ref{equ: clp-mea}) in non-critical slope case (i.e. $\val_p(\alpha)<k+1$) was proved in \cite[Prop. 4.9]{Em05}, and the critical slope (i.e. $\val_p(\alpha)=k+1$) but non-critical case follows by the same argument.
\end{proof}
\begin{remark} By results of Amice-V\'elu and Vishik (e.g. see \cite[Lem. 3.14]{Em05}), when ${z_{\alpha}}$ is not of critical slope, $L^{\pm}(z_{\alpha},-)$ is determined (up to non-zero scalars) by the interpolation property.
\end{remark}
\begin{proposition}\label{cri0}
  If $z_{\alpha}$ is critical, then $L^{\pm}(z_{\alpha},\phi x^j)=0$ for all $\phi x^j$ given as in (\ref{equ: clp-int}).
\end{proposition}
\begin{proof}
We denote by $\cC^{\lp,\leq k}(N(\Z_p^{\times}), E)$ the closed subspace of $\cC^{\la}(N(\Z_p^{\times}), E)$ consisting  of functions that are locally polynomial of degree $\leq k$. We have a natural commutative diagram (e.g. see \cite[(3.15) (3.16)]{Em05}):
\begin{equation*}
  \begin{CD}
    \cC^{\lp,\leq k}(N(\Z_p^{\times}), E) @>>> (\Ind_{\overline{B}(\Q_p)}^{\GL_2(\Q_p)} \chi_{z_{\alpha}}^{\sm}\delta_B^{-1})^{\infty}\otimes_E (\Sym^k E^2)^{\vee}  \\
    @VVV @VVV \\
    \cC^{\la}(N(\Z_p^{\times}), E) @>>> \big(\Ind_{\overline{B}(\Q_p)}^{\GL_2(\Q_p)} \chi_{z_{\alpha}} \delta_B^{-1}\big)^{\an}
  \end{CD}
\end{equation*}
where the horizontal maps are injective and are given as in the first map in (\ref{equ: clp-mea}).
The commutative diagram, together with Proposition \ref{prop: clp-rest}, imply that $L^{\pm}(z_{\alpha}, \psi)=0$ for all $\psi \in \cC^{\lp,\leq k}(N(\Z_p^{\times}), E) $. It is easy to see that all the $\phi x^j$ (given as in (\ref{equ: clp-int})) lie in $\cC^{\lp,\leq k}(N(\Z_p^{\times}), E) $. The proposition follows.
\end{proof}
\begin{remark}
 (1) It is  not clear to the author whether $L^{\pm}(z_{\alpha}, -)$ is zero or not.

 (2) We can also directly construct the critical $p$-adic $L$-functions $L^{\pm}(z_{\alpha},-)$ without using the two-variable $p$-adic $L$-functions $L^{\pm}(-,-)$. In fact, we have:
 \begin{multline}\label{mult1}
\dim_E   \Hom_{\GL_2(\Q_p)}\big((\Ind_{\overline{B}(\Q_p)}^{\GL_2(\Q_p)} \chi_{z_{\alpha}}^c\delta_B^{-1})^{\an}, \widetilde{H}^1_{\et,c}(K^p,E)_{\overline{\rho}}^{\an,\pm}[\cH^p=\lambda_f]\big) \\
=\dim_E  \Hom_{\GL_2(\Q_p)}\big( (\Ind_{\overline{B}(\Q_p)}^{\GL_2(\Q_p)} \chi_{z_{\alpha}}^{\sm}\delta_B^{-1})^{\infty}\otimes_E (\Sym^k E^2)^{\vee}, \widetilde{H}^1_{\et,c}(K^p,E)_{\overline{\rho}}^{\an,\pm}[\cH^p=\lambda_f]\big)=1
 \end{multline}
where the first equation is a consequence of the local-global compatibility of $p$-adic Langlands correspondence \cite[Thm. 1.2.1]{Em4}, and the second equation follows from the multiplicity one result.
Any non-zero element $j$ in the $1$-dimensional $E$-vector space on the left hand side of (\ref{mult1}) induces (where the first map is the same as the first map in (\ref{equ: clp-mea}))
 \begin{equation}\label{padicL2}
\cC^{\la}(\Z_p^{\times}, E)\hooklongrightarrow  \big(\Ind_{\overline{B}(\Q_p)}^{\GL_2(\Q_p)} \chi_{z_{\alpha}} \delta_B^{-1}\big)^{\an}
\twoheadlongrightarrow (\Ind_{\overline{B}(\Q_p)}^{\GL_2(\Q_p)} \chi_{z_{\alpha}}^c\delta_B^{-1})^{\an} \xlongrightarrow{j}  \widetilde{H}^1_{\et,c}(K^p,E)_{\overline{\rho}}^{\an,\pm} \xlongrightarrow{(\ref{equ: clp-cpa})} E.
 \end{equation}
By (\ref{mult1}), $j$ is equal to the second morphism in (\ref{compan}) up to non-zero scalars. We deduce  that (\ref{padicL2}) is equal to $L^{\pm}(z_{\alpha},-)$ up to non-zero scalars. We also remark that this construction does not rely on the smoothness of the eigencurve.
\end{remark}

\subsection{Properties}\label{sec: 22}Keep the notation in \S~\ref{sec: 21}, and assume $z_{\alpha}$ is critical. In this section, we study some properties of $L^{\pm}(z_{\alpha},-)$ and the two-variable $p$-adic $L$-functions $L^{\pm}(-,-)$.
\subsubsection{Relations with Bella\"iche's critical $p$-adic $L$-functions}
Recall in \cite{Bel}, Bellaiche constructed $2$-variable $\pm$-$p$-adic $L$-functions $\sL^{\pm}$ in a neighborhood of $z_{\alpha}$. In this section, we compare $\sL^{\pm}$ with our $L^{\pm}$ constructed in \S~\ref{sec: 21}.

Let $\cV\subseteq \cU$ be an affinoid open neighborhood of $z_{\alpha}$ in $\cC$ such that both $L^{\pm}$ and $\sL^{\pm}$ are defined, that the points of non-critical slope are Zariski-dense in $\cV$, and that $\co(\cV)$ is a PID.

The isomorphism $\Z_p^{\times} \cong (\Z/q \Z)^{\times}\times (1+q \Z_p)$ (with $q=p$ if $p\neq 2$, and $q=2p$ if $p=2$) induces an isomorphism of rigid spaces $\cW\cong \sqcup_{i\in (\Z/q \Z)^{\times}} \cW_i$ where all the $\cW_i$ are isomorphic to the rigid space over $E$ parametrizing continuous characters of $1+q\Z_p\cong \Z_p$. Note that the latter rigid space is isomorphic to the open unit disc: a point $z$ of $\bC_p$ with $|z|<1$ corresponds to the character $(1+z)^a$ for $a\in \Z_p$. For $\lambda\in \co(\cW)$ (resp. $\Lambda\in \co(\cV \times \cW)$, we denote by $\lambda^i\in \co(\cW_i)$ (resp. $\Lambda^i\in \co(\cV \times \cW_i)$) its restriction on $\cW_i$ (resp. on $\co(\cV \times \cW_i)$).

\begin{proposition}Let $i\in (\Z/q \Z)^{\times}$. Assume $\sL^{\pm}(z_{\alpha},-)^i\neq 0$ (resp. $L^{\pm}(z_{\alpha}, -)^i\neq 0$), then there exist an admissible affinoid $\cV'\subset \cV$ containing $z_{\alpha}$ and $a_{\pm}\in \co(\cV')$ such that $L^{\pm}(-,-)^i=a_{\pm} \sL^{\pm}(-,-)^i$ (resp. $\sL^{\pm}(-,-)^i=a_{\pm} L^{\pm}(-,-)^i$).
\end{proposition}
\begin{proof}
We only prove the case where $\sL^{\pm}(z_{\alpha},-)^i\neq 0$, the other case being symmetric.
Let $w_1, w_2\in \cW_i$, put: \begin{equation*}d_{w_1,w_2}^{\pm}(-):=L^{\pm}(-,w_1)\sL^{\pm}(-,w_2)-L^{\pm}(-,w_2)\sL^{\pm}(-,w_1)\in \co(\cV).
\end{equation*}
We claim $d_{w_1,w_2}^{\pm}=0$ (thus independent of $w_1,w_2$). Indeed for any  point $z\in \cV$ of non-critical slope, we know the distributions $L^{\pm}(z,-)$, $\sL^{\pm}(z,-)$ equal up to non-zero scalars (by the interpolation property), thus $d_{w_1,w_2}^{\pm}(z)=0$. Since such points are Zariski-dense in $\cV$, the claim follows.

For $w\in \cW_i$ such that $\sL^{\pm}(z_{\alpha}, w)=\sL^{\pm}(z_{\alpha}, w)^i\neq 0$ (by assumption, such $w$ exists), we put $a_{\pm}':=\frac{L^{\pm}(-,w)}{\sL^{\pm}(-,w)}\in \Fra (\co(\cV))$. By the above claim, we see $a_{\pm}'$  is independent of the choice of $w$. We put  $a_{\pm}:=\frac{L^{\pm}(-,-)^i}{\sL^{\pm}(-,-)^i}\in \Fra(\co(\cV \times \cW_i))$. We claim $a_{\pm}'=a_{\pm}$ (in other words $a_{\pm}\in \Frac(\co(\cV))$). We view $a_{\pm}'$ as an element in $\Fra (\co(\cV \times \cW_i))$ by the natural inclusion. To prove $a_{\pm}'=a_{\pm}$, it is sufficient to prove
\begin{equation*}\mathfrak{d}:=\sL^{\pm}(-,w)L^{\pm}(-,-)^i-L^{\pm}(-,w)\sL^{\pm}(-,-)^i=0
\end{equation*}
 where $w$ satisfies $\sL^{\pm}(z_{\alpha},w)\neq 0$,  $\sL^{\pm}(-,w)$ and $L^{\pm}(-,w)$ are viewed as elements in $\co(\cV \times\cW_i)$ via the natural injection $\co(\cV)\hookrightarrow \co(\cV \times \cW_i)$ (and so $\mathfrak{d}\in \co(\cV \times\cW_i)$). Let $Z$ be the set of classical points of non-critical slope in $\cV$. As discussed above, we know
 \begin{equation*}\mathfrak{d}(z,w')=\sL^{\pm}(z,w)L^{\pm}(z,w')^i-L^{\pm}(z,w) \sL^{\pm}(z,w')^i=0
  \end{equation*}for all $z\in Z$, and $w'\in \cW_i$. Since $Z\times \cW_i$ is Zariski-dense in $\cV \times\cW_i$, we deduce $\mathfrak{d}=0$. Thus $a_{\pm}=a_{\pm}'\in \Fra(\co(\cV))$. The proposition follows by shrinking $\cV$ (and using the assumption $\sL^{\pm}(z_{\alpha},-)^i\neq 0$).
\end{proof}
Recall that the restriction $\sL^{\pm}(z_{\alpha},-)$ is equal, up to non-zero scalars, to the $\pm$-$p$-adic $L$-function for $f_{\alpha}$ constructed in \cite{Bel}, we can thus deduce from the proposition the following corollary.
\begin{corollary}Let $i\in (\Z/q\Z)^{\times}$. Assume $\sL^{\pm}(z_{\alpha},-)^i\neq 0$ (resp. $L^{\pm}(z_{\alpha},-)^i\neq 0$), then there exists $a^{\pm}\in E$ (enlarge $E$ if necessary) such that $L^{\pm}(z_{\alpha},-)^i=a^{\pm} \sL^{\pm}(z_{\alpha}, -)^i$ (resp. $\sL^{\pm}(z_{\alpha},-)^i=a^{\pm} L^{\pm}(z_{\alpha}, -)^i$).
\end{corollary}
\subsubsection{Secondary critical $p$-adic $L$-functions}
For $i=1,\cdots, e-1$, we put (shrinking $\cU$ if necessary) $L^{\pm}_{i}\in \co(\cU \times \cW)$ such that (noting that $r_{z_{\alpha}}$ is a uniformiser of $\co(\cU)$ at $z_{\alpha}$)
\begin{equation*}L^{\pm}_{i}(t, \sigma):=\frac{\partial^i L^{\pm}}{\partial r_{z_{\alpha}}^{i}}(t,\sigma), \forall \ (t,\sigma)\in \cU \times \cW.
\end{equation*}
The statement in the following proposition was proved in \cite[\S~4.4]{Bel} for Bella\"iche's two-variable $p$-adic $L$-functions. We show that it also holds for our $L^{\pm}(-,-)$.
\begin{proposition}
  (1) For $i=1,\cdots, e-2$, $L^{\pm}_i(z_{\alpha}, \phi x^j)=0$ for all $\phi x^j$ given as in (\ref{equ: clp-int}).

  (2) With the notation of (\ref{equ: clp-int}), we have up to non-zero scalars (independent of $\phi x^j$):
\begin{equation}\label{int2}
  L^{\pm}_{e-1}(z_{\alpha},\phi x^j)=e_p(\alpha,\phi x^j)\frac{m^{j+1}}{(-2\pi i)^j}\frac{j!}{\tau(\phi^{-1})}\frac{L_{\infty}(f\phi^{-1},j+1)}{\Omega_f^{\pm}}.
\end{equation}
\end{proposition}
\begin{proof}
Step (a). We unwind a bit Emerton's adjunction formula (\cite[Thm. 0.13]{Em2}).
By \cite[Thm. 3.5.6]{Em11}, the composition in (\ref{equ: clp-cEp}) first induces $B(\Q_p)$-equivariant morphisms
\begin{equation}\label{adj000}
  \cC^{\sm}_c(N(\Q_p), \cM^{\pm}(\cU)^{\vee}_b \otimes_E \delta_B^{-1}) \xlongrightarrow{r_{z_{\alpha}}^{e-1}} \cC_c^{\sm}(N(\Q_p), \cM^{\pm}(\cU)^{\vee}_b\otimes_E \delta_B^{-1})\xlongrightarrow{\iota_0} \widetilde{H}^1_{\et,c}(K^p,E)_{\overline{\rho}}^{\an, \pm},
\end{equation}
where $\cC^{\sm}_c(N(\Q_p),\cM^{\pm}(\cU)^{\vee}_b \otimes_E \delta_B^{-1})$ denotes the space of locally constant, compactly supported functions on $N(\Q_p)$ with values in $\cM^{\pm}(\cU)^{\vee}_b \otimes_E \delta_B^{-1}$ (which is equipped with a natural $B(\Q_p)$-action as in \cite[\S~3.5]{Em11}). The morphisms in (\ref{adj000}) further induce $(\ug, B(\Q_p))$-equivariant morphisms (cf. \cite[(5.11)]{Em07})
\begin{multline}\label{adj001}
 \text{U}(\ug) \otimes_{\text{U}(\ub)} \cC^{\sm}_c(N(\Q_p), \cM^{\pm}(\cU)^{\vee}_b \otimes_E \delta_B^{-1})\xlongrightarrow{r_{z_{\alpha}}^{e-1}}\text{U}(\ug) \otimes_{\text{U}(\ub)} \cC_c^{\sm}(N(\Q_p), \cM^{\pm}(\cU)^{\vee}_b\otimes_E \delta_B^{-1})\\ \xlongrightarrow{\iota_1^{\pm}} \widetilde{H}^1_{\et,c}(K^p,E)_{\overline{\rho}}^{\an, \pm}.
\end{multline}
On the other hand, we have a natural $(\ug, B(\Q_p))$-equivariant morphism (cf. \cite[(2.8.7)]{Em2})
\begin{multline}\label{Verma}
 \text{U}(\ug) \otimes_{\text{U}(\ub)} \cC^{\sm}_c(N(\Q_p), \cM^{\pm}(\cU)^{\vee}_b \otimes_E \delta_B^{-1}) \lra \cC^{\lp}_c(N(\Q_p), \cM^{\pm}(\cU)^{\vee}_b \otimes_E \delta_B^{-1}) \\
 \cong \cC^{\sm}_c(N(\Q_p, \cM^{\pm}(\cU)^{\vee}_b \otimes_E \delta_B^{-1}) \otimes_E E[z],
\end{multline}
where $\cC^{\lp}_c(N(\Q_p), \cM^{\pm}(\cU)^{\vee}_b \otimes_E \delta_B^{-1})$ denotes the space of locally polynomial functions on $N(\Q_p)$ with values in $\cM^{\pm}(\cU)^{\vee}_b \otimes_E \delta_B^{-1}$ and we refer to \cite[Def. 2.5.21]{Em2} for the precise definition and for the second isomorphism in (\ref{Verma}). As in \cite[(4.5.16)]{Em1},
the morphism in (\ref{Verma}) is given by
\begin{equation*}X_-^{l} \otimes f \mapsto ((\prod_{i=0}^{l-1} (h-i)) \cdot f) z^l
 \end{equation*}
 for $l\in \Z_{\geq 1}$, and $f\in \cC^{\sm}_c(N(\Q_p), \cM^{\pm}(\cU)^{\vee}_b \otimes_E \delta_B^{-1})$ (and $X_-^0 \otimes f \mapsto f$). As in the last paragraph of the proof of \cite[Lem. 4.5.12]{Em1}, the first morphism in (\ref{Verma}) is surjective. 
Since the composition in (\ref{equ: clp-cEp}) is balanced, the composition in (\ref{adj001}) factors through (\ref{Verma}) (noting that in contrary $\iota_1^{\pm}$ does not factor through (\ref{Verma})). In summary, we have a $(\ug, B(\Q_p))$-equivariant commutative diagram
\begin{equation}\label{comt00}
\begin{CD}
  \text{U}(\ug) \otimes_{\text{U}(\ub)} \cC^{\sm}_c(N(\Q_p), \cM^{\pm}(\cU)^{\vee}_b \otimes_E \delta_B^{-1})@> r_{z_{\alpha}}^{e-1}>> \text{U}(\ug) \otimes_{\text{U}(\ub)} \cC^{\sm}_c(N(\Q_p), \cM^{\pm}(\cU)^{\vee}_b\otimes_E \delta_B^{-1}) \\
  @VVV @V\iota_1^{\pm} VV \\
  \cC^{\lp}_c(N(\Q_p), \cM^{\pm}(\cU)^{\vee}_b \otimes_E \delta_B^{-1}) @>>> \widetilde{H}^1_{\et,c}(K^p,E)_{\overline{\rho}}^{\an, \pm}
\end{CD}.
\end{equation}
We remark that the morphism (\ref{equ: clp-ppA}) is actually induced by the bottom map in (\ref{comt00}) (e.g. see \cite[Cor. 4.3.3]{Em2}).
For $N(\Z_p^{\times})=\begin{pmatrix} 1 & \Z_p^{\times} \\ 0 & 1\end{pmatrix}$, we have a natural injection $\cC^*(N(\Z_p^{\times}), -) \hookrightarrow \cC^*_c(N(\Q_p),-)$,  sending a function $F$ to the function whose value at $x\in N(\Z_p^{\times})$ is $F(x)$ and $0$ outside  $N(\Z_p^{\times})$.  We then easily deduce from (\ref{comt00}) a commutative diagram (which is $\text{U}(\ug)$-equivariant)
\begin{equation}\label{diagK}
\begin{CD}
  \text{U}(\ug) \otimes_{\text{U}(\ub)} \cC^{\sm}(N(\Z_p^{\times}), \cM^{\pm}(\cU)^{\vee}_b \otimes_E \delta_B^{-1})@> r_{z_{\alpha}}^{e-1}>> \text{U}(\ug) \otimes_{\text{U}(\ub)} \cC^{\sm}(N(\Z_p^{\times}), \cM^{\pm}(\cU)^{\vee}_b\otimes_E \delta_B^{-1}) \\
  @VVV @V\iota_1^{\pm} VV \\
  \cC^{\lp}(N(\Z_p^{\times}), \cM^{\pm}(\cU)^{\vee}_b \otimes_E \delta_B^{-1}) @>>> \widetilde{H}^1_{\et,c}(K^p,E)_{\overline{\rho}}^{\an, \pm}
\end{CD}.
\end{equation}
The bottom map of (\ref{diagK}) is in fact equal to the composition
\begin{multline*}
\cC^{\lp}(N(\Z_p^{\times}), \cM^{\pm}(\cU)^{\vee}_b \otimes_E \delta_B^{-1}) \hooklongrightarrow \cC^{\la}(N(\Z_p^{\times}), \cM^{\pm}(\cU)^{\vee}_b \otimes_E \delta_B^{-1})\\  \lra
\big(\Ind_{\overline{B}(\Q_p)}^{\GL_2(\Q_p)} \cM^{\pm}(\cU)^{\vee}_b\otimes_E \delta_B^{-1}\big)^{\an} \xlongrightarrow{\iota^{\pm}} \widetilde{H}^1_{\et,c}(K^p,E)_{\overline{\rho}}^{\an, \pm},
\end{multline*}
where the middle map is given in the same way as in the discussion below  (\ref{padicL}).
\\

Step (b). Let $\cV$ be a compact open subset of $\Z_p^{\times}$, and $1_{\cV}\in \cC^{\infty}(\Z_p^{\times},E)$ be the function with $1_{\cV}(x)=\begin{cases}
  1 & x\in \cV \\ 0 & \text{otherwise}
\end{cases}$. For $j\in \Z_{\geq 0}$ and $i=1,\cdots, e-1$, we have
\begin{equation}\label{deri0}
L^{\pm}_i(-,1_{\cV} z^j)=\frac{d L^{\pm}(-,1_{\cV} z^j)}{d^i r_{z_{\alpha}}}.
\end{equation}Recall that we have fixed an isomorphism $\co(\cU)\cong \cM^{\pm}(\cU)$, which induces an isomorphism $\co(\cU)^{\vee}_b \cong \cM^{\pm}(\cU)^{\vee}_b\otimes_E \delta_B^{-1}$, that we fix in the sequel.
By definition, $L^{\pm}(-,1_{\cV} z^j)\in \co(\cU)$ is characterized by the following composition
\begin{equation}\label{Lfunval}
  \co(\cU)^{\vee}_b \cong \cM^{\pm}(\cU)^{\vee}_b \otimes_E \delta_B^{-1} \xlongrightarrow{h_{\cV, j}} \cC^{\lp}(N(\Z_p^{\times}), \cM^{\pm}(\cU)^{\vee}_b \otimes_E \delta_B^{-1}) \\ \lra \widetilde{H}^1_{\et,c}(K^p,E)_{\overline{\rho}}^{\an, \pm} \xlongrightarrow{\{\infty\}-\{0\}} E,
\end{equation}
where  the map $h_{\cV,j}$ sends $m \in  \cM^{\pm}(\cU)^{\vee}_b \otimes_E \delta_B^{-1}$ to the element:
\begin{equation*}m\otimes 1_{\cV} \otimes z^j \in \cC^{\sm}(N(\Z_p^{\times}), E) \otimes_E \cM^{\pm}(\cU)^{\vee}_b \otimes_E \delta_B^{-1}) \otimes_E E[z]
\cong \cC^{\lp}(N(\Z_p^{\times}), \cM^{\pm}(\cU)^{\vee}_b \otimes_E \delta_B^{-1}).
\end{equation*} Denote by $F_{\cV, j}\in \co(\cU)$ the element given by the following composition
\begin{multline}\label{Fvj}
  \co(\cU)^{\vee}_b \cong \cM^{\pm}(\cU)^{\vee}_b \otimes_E \delta_B^{-1} \xlongrightarrow{g_{\cV,j}} \text{U}(\ug) \otimes_{\text{U}(\ub)} \cC^{\sm}(N(\Z_p^{\times}), \cM^{\pm}(\cU)^{\vee}_b \otimes_E \delta_B^{-1})\\  \xlongrightarrow{\iota_1^{\pm}\circ r_{z_{\alpha}}^{e-1}} \widetilde{H}^1_{\et,c}(K^p,E)_{\overline{\rho}}^{\an, \pm}\xlongrightarrow{\{\infty\}-\{0\}}  E,
\end{multline}
where $g_{\cV,j}$ sends $m$ to $X_-^{j} \otimes 1_{\cV} \otimes m$.
By the description of (\ref{Verma}), we see
\begin{equation}\label{fomul1}F_{\cV,j}=\Delta_j L^{\pm}(-,1_{\cV} z^j)
\end{equation}with $\Delta_j=\prod_{l=0}^{j-1}(h-l)$ for $j\in \Z_{>0}$, and $\Delta_0=1$. Similarly, the composition
\begin{multline}\label{Gvj}
  \co(\cU)^{\vee}_b \cong \cM^{\pm}(\cU)^{\vee}_b \otimes_E \delta_B^{-1} \xlongrightarrow{g_{\cV,j}} \text{U}(\ug) \otimes_{\text{U}(\ub)} \cC^{\sm}(N(\Z_p^{\times}), \cM^{\pm}(\cU)^{\vee}_b \otimes_E \delta_B^{-1})\\ \xlongrightarrow{\iota_1^{\pm}} \widetilde{H}^1_{\et,c}(K^p,E)_{\overline{\rho}}^{\an, \pm} \xlongrightarrow{\{\infty\}-\{0\}}  E
\end{multline}
gives an element $G_{\cV,j}\in \co(\cU)$ as well. Comparing  (\ref{Fvj}) with (\ref{Gvj}), we see
\begin{equation}\label{fomul2}F_{\cV,j}=G_{\cV,j} r_{z_{\alpha}}^{e-1}.
 \end{equation}
Suppose $j\leq k$, then $\Delta_j$ and $r_{z_{\alpha}}$ are coprime (using $(r_{z_{\alpha}}^{e})=(h-k)$). By (\ref{fomul1}) and (\ref{fomul2}),  we deduce (shrinking $\cU$ if necessary) that there exists $H_{\cV,j}\in \co(\cU)$ such that $L^{\pm}(-,1_{\cV} z^j)=r_{z_{\alpha}}^{e-1} H_{\cV,j}$ and $G_{\cV,j}=\Delta_j H_{\cV,j}$. Together with (\ref{deri0}), the part (1) follows and we have $L_{e-1}^{\pm}(z_{\alpha},1_{\cV}z^j)=(e-1)!H_{\cV,j}(z_{\alpha})$.
\\

Step (c). We prove the part (2). We fix an isomorphism of $E$-vector space $\chi_{z_{\alpha}}\cong E$. Denote by $i_{z_{\alpha}}: \co(\cU)\twoheadrightarrow \co(\cU)/r_{z_{\alpha}}=E$ the natural projection, which induces $i_{z_{\alpha}}^*:  \chi_{z_{\alpha}} \hookrightarrow \co(\cU)^{\vee}_b$. We remark that for $s\in \co(\cU)$ the evaluation $s(z_{\alpha})$ is given by $i_{z_{\alpha}}(s)$. Since $i_{z_{\alpha}}^*$ has image in $\cM^{\pm}(\cU)^{\vee}_b[h=k]^{\cl}$ (by (\ref{chaofcl})), as in \cite[Ex. 5.22]{Em07}, the composition
\begin{equation*}
  \chi_{z_{\alpha}} \hookrightarrow \co(\cU)^{\vee}_b\cong \cM^{\pm}(\cU)^{\vee}_b  \lra J_B\big(\widetilde{H}^1_{\et,c}(K^p,E)_{\overline{\rho}}^{\an, \pm}\big)
\end{equation*}
is balanced and induces an injection
\begin{equation*}
  \iota_{z_{\alpha}}^{\pm}:   (\Ind_{\overline{B}(\Q_p)}^{\GL_2(\Q_p)} \chi_{z_{\alpha}}^{\sm}\delta_B^{-1})^{\infty}\otimes_E (\Sym^k E^2)^{\vee}  \hooklongrightarrow  \widetilde{H}^1_{\et,c}(K^p,E)_{\overline{\rho}}^{\an,\pm}.
\end{equation*}
We have hence a commutative diagram
\begin{equation}\label{comt01}
  \begin{CD}
     \text{U}(\ug) \otimes_{\text{U}(\ub)} \cC^{\sm}(N(\Z_p^{\times}), \chi_{z_{\alpha}} \otimes_E \delta_B^{-1}) @> i_{z_{\alpha}}^*>>    \text{U}(\ug) \otimes_{\text{U}(\ub)} \cC^{\sm}_c(N(\Z_p^{\times}), \co(\cU)^{\vee}_b\otimes_E \delta_B^{-1})\\
     @VVV @V \iota_1^{\pm} VV \\
     (\Ind_{\overline{B}(\Q_p)}^{\GL_2(\Q_p)} \chi_{z_{\alpha}}^{\sm})^{\infty}\otimes_E (\Sym^k E^2)^{\vee} @> \iota_{z_{\alpha}}^{\pm} >>  \widetilde{H}^1_{\et,c}(K^p,E)_{\overline{\rho}}^{\an,\pm}
  \end{CD}
\end{equation}
where the top morphism is induced by $i_{z_{\alpha}}^*$, and  the left hand side map factors as
\begin{multline*}
\text{U}(\ug) \otimes_{\text{U}(\ub)} \cC^{\sm}(N(\Z_p^{\times}), \chi_{z_{\alpha}} \otimes_E \delta_B^{-1})  \lra \cC^{\lp, \leq k}(N(\Z_p^{\times}), \chi_{z_{\alpha}} \otimes_E \delta_B^{-1})\\
\hooklongrightarrow  (\Ind_{\overline{B}(\Q_p)}^{\GL_2(\Q_p)} \chi_{z_{\alpha}}^{\sm}\delta_B^{-1})^{\infty}\otimes_E (\Sym^k E^2)^{\vee} \Big(\hooklongrightarrow   (\Ind_{\overline{B}(\Q_p)}^{\GL_2(\Q_p)} \chi_{z_{\alpha}}\delta_B^{-1})^{\an}\Big),
\end{multline*}
with the first map given in the same way as in (\ref{Verma}), and the second map given in the same way as in the discussion below  (\ref{padicL}).
By (\ref{comt01}), $G_{\cV,j}(z_{\alpha})$ is equal to  the evaluation at $1\in E$ of the following composition
\begin{equation*}
  E\cong \chi_{z_{\alpha}} \otimes_E \delta_B^{-1} \xrightarrow{g_{\cV,j}} \text{U}(\ug) \otimes_{\text{U}(\ub)} \cC^{\sm}(N(\Z_p^{\times}), \chi_{z_{\alpha}} \otimes_E \delta_B^{-1}) \\
  \xlongrightarrow{\iota_1^{\pm}\circ i_{z_{\alpha}}^*}  \widetilde{H}^1_{\et,c}(K^p,E)_{\overline{\rho}}^{\an,\pm} \xlongrightarrow{\{\infty\}-\{0\}} E.
\end{equation*}
Consider the following composition (recall $j\leq k$, and the map $h_{\cV,j}$ is given in the same way as in (\ref{Lfunval}))
\begin{multline*}
\mu_{\cV,j}: E\cong \chi_{z_{\alpha}} \otimes_E \delta_B^{-1} \xlongrightarrow{h_{\cV, j}} \cC^{\lp, \leq k}(N(\Z_p^{\times}), \chi_{z_{\alpha}} \otimes_E \delta_B^{-1})\\ \hooklongrightarrow  (\Ind_{\overline{B}(\Q_p)}^{\GL_2(\Q_p)} \chi_{z_{\alpha}}^{\sm}\delta_B^{-1})^{\infty}\otimes_E (\Sym^k E^2)^{\vee} \xlongrightarrow{\iota_{z_{\alpha}}^{\pm}} \widetilde{H}^1_{\et,c}(K^p,E)_{\overline{\rho}}^{\an,\pm} \xlongrightarrow{\{\infty\}-\{0\}} E.
\end{multline*}
By the commutative diagram (\ref{comt01}) (and using the description of (\ref{Verma})),  we see
\begin{equation*}G_{\cV,j}(z_{\alpha})=\Delta_j(z_{\alpha}) \mu_{\cV,j}(1)
\end{equation*}
and hence (by step (b)) $L_{e-1}^{\pm}(z_{\alpha}, 1_{\cV}z^j)=(e-1)! \mu_{\cV,j}(1)$.  In summary, the values of $L_{e-1}^{\pm}(z_{\alpha},-)$ at functions in $\cC^{\lp, \leq k} (\Z_p^{\times},E)$ (in particular, at $\phi x^j$ as in (\ref{equ: clp-int})) are characterized by the composition
\begin{multline*}
  \cC^{\lp, \leq k}(N(\Z_p^{\times}), \chi_{z_{\alpha}} \otimes_E \delta_B^{-1}) \hooklongrightarrow  (\Ind_{\overline{B}(\Q_p)}^{\GL_2(\Q_p)} \chi_{z_{\alpha}}^{\sm}\delta_B^{-1})^{\infty}\otimes_E (\Sym^k E^2)^{\vee} \\ \xlongrightarrow{\iota_{z_{\alpha}}^{\pm}} \widetilde{H}^1_{\et,c}(K^p,E)_{\overline{\rho}}^{\an,\pm} \xlongrightarrow{\{\infty\}-\{0\}} E.
\end{multline*}
Now we are in the same situation as in \cite[(4.14)(4.15)]{Em05}. The proof of \cite[Prop. 4.9]{Em05} carries over verbatim to this setting, and the part (2) follows.
\end{proof}

\end{document}